\documentclass[11pt]{article}
\usepackage{amssymb}
\usepackage{amsmath}
\usepackage{amsthm}
\usepackage{enumerate}
\usepackage{color}
\usepackage[dvipsnames]{xcolor}
\usepackage{fullpage}
\usepackage{hyperref}

\usepackage[utf8]{inputenc}
\usepackage{frcursive}

\usepackage{graphicx, float}

\newtheorem{theorem}{Theorem}[section]

\newtheorem{corollary}[theorem]{Corollary}
\newtheorem{lemma}[theorem]{Lemma}

\newtheorem*{claim*}{Claim}
\newtheorem{conjecture}[theorem]{Conjecture}
\newtheorem{proposition}[theorem]{Proposition}

\newtheorem{remark}[theorem]{Remark}

\newtheorem{qn}[theorem]{Question}

\newtheorem{theorem LSV}{Theorem LSV}
\newtheorem*{theorem LSV*}{Theorem LSV}

\newtheorem{theorem MTP}{Mass Transference Principle}
\newtheorem*{theorem MTP*}{Mass Transference Principle}
\newtheorem{theorem K}{Khintchine's Theorem}
\newtheorem*{theorem K*}{Khintchine's Theorem}
\newtheorem{theorem J}{Jarn\'{\i}k's Theorem}
\newtheorem*{theorem J*}{Jarn\'{\i}k's Theorem}
\newtheorem{theorem KJ}{Khintchine--Jarn\'{\i}k Theorem}
\newtheorem*{theorem KJ*}{Khintchine--Jarn\'{\i}k Theorem}
\newtheorem{theorem BV1}{Theorem BV1}
\newtheorem*{theorem BV1*}{Theorem BV1}
\newtheorem{theorem BV2}{Theorem BV2}
\newtheorem*{theorem BV2*}{Theorem BV2}
\newtheorem{theorem KG}{Theorem KG}
\newtheorem*{theorem KG*}{Theorem KG}
\newtheorem{theorem IHKG}{Inhomogeneous Khintchine--Groshev Theorem}
\newtheorem*{theorem IHKG*}{Inhomogeneous Khintchine--Groshev Theorem}
\newtheorem{theorem DLN1}{Theorem DLN1}
\newtheorem*{theorem DLN1*}{Theorem DLN1}
\newtheorem{theorem DLN2}{Theorem DLN2}
\newtheorem*{theorem DLN2*}{Theorem DLN2}
\newtheorem{theorem DLN3}{Theorem DLN3}
\newtheorem*{theorem DLN3*}{Theorem DLN3}
\newtheorem{theorem S}{Theorem S}
\newtheorem*{theorem S*}{Theorem S}

\numberwithin{equation}{section}

\renewcommand{\Bbb}[1]{\mathbb{#1}}
\newcommand{\N}{{\Bbb N}}         

\newcommand{\R}{{\Bbb R}}        
\newcommand{\Z}{{\Bbb Z}}         

\newcommand{\cA}{{\cal A}}

\newcommand{\cC}{{\cal C}}

\newcommand{\cH}{{\cal H}}

\newcommand{\cL}{{\cal L}}

\newcommand{\cN}{{\cal N}}

\newcommand{\cR}{{\cal R}}
\newcommand{\cS}{{\cal S}}

\renewcommand{\le}{\leq}
\renewcommand{\ge}{\geq}

\DeclareMathOperator{\dimh}{dim_H}

\newcommand{\mmod}[1]{\,\,\mathrm{mod}\,\,#1}

\def\alp{{\alpha}} 
\def\bet{{\beta}}  
\def\gam{{\gamma}} 
\def\del{{\delta}} \def\Del{{\Delta}}

\def\tet{{\theta}}  

 \def\Lam{{\Lambda}}

\def\eps{\varepsilon}

\def\le{\leqslant} \def\ge{\geqslant}

\def\d{{\,{\rm d}}}

\def \leq {\le}
\def \geq {\ge}

\def \bC {\mathbb C}

\def \bN {\mathbb N}

\def \bQ {\mathbb Q}
\def \bR {\mathbb R}
\def \bZ {\mathbb Z}

\def \cA {\mathcal A}

\def \cC {\mathcal C}

\def \cH {\mathcal H}

\def \cL {\mathcal L}

\def \cN {\mathcal N}

\def \cR {\mathcal R}
\def \cS {\mathcal S}

\def \supp {{\mathrm{supp}}}

\newcommand{\Addresses}{{
  \bigskip
  \footnotesize

  D.~Allen, 
  \textsc{School of Mathematics, University of Bristol, Fry Building 
Woodland Road, Bristol, BS8 1UG, and the Heilbronn Institute for Mathematical Research, Bristol, UK}\par\nopagebreak
  \textit{E-mail address:} \texttt{demi.allen@bristol.ac.uk}

  \medskip

  S.~Chow, 
  \textsc{Mathematics Institute, Zeeman Building, University of Warwick, Coventry, CV4 7AL, UK}\par\nopagebreak
  \textit{E-mail address:} \texttt{sam.chow@warwick.ac.uk}
  
  \medskip
  
  H.~Yu, 
  \textsc{Department of Pure Mathematics and Mathematical Statistics, Centre for Mathematical Sciences, Cambridge, CB3 0WB, UK}\par\nopagebreak
  \textit{E-mail address:} \texttt{hy351@maths.cam.ac.uk}
}}

\title{Dyadic Approximation in the Middle-Third Cantor Set}
\author{Demi Allen\footnote{Supported by the Heilbronn Institute for Mathematical Research} \and Sam Chow \footnote{Supported by EPSRC Fellowship Grant EP/S00226X/2} \and Han Yu \footnote{Supported by the European Research Council (ERC) under the European Union’s Horizon 2020 research and innovation programme (grant agreement No. 803711), and indirectly by Corpus Christi College, Cambridge}}

\date{\today}

\date{\small{\itshape Dedicated to Professor Sanju Velani}}

\begin{document}

\maketitle
\begin{abstract}
In this paper, we study the metric theory of dyadic approximation in the middle-third Cantor set. This theory complements earlier work of Levesley, Salp, and Velani (2007), who investigated the problem of approximation in the Cantor set by triadic rationals. We find that the behaviour when we consider dyadic approximation in the Cantor set is substantially different to considering triadic approximation in the Cantor set. In some sense, this difference in behaviour is a manifestation of Furstenberg's times 2 times 3 phenomenon from dynamical systems, which asserts that the base 2 and base 3 expansions of a number are not both structured. 
\end{abstract}


\section{Introduction}

\subsection{Background and statement of results}

In 1984, Mahler wrote a note entitled ``Some Suggestions for Further Research''~\cite{Mahler84} in which he posed a number of interesting questions which he deemed worthy of attention. One of these questions posed by Mahler, which is of particular interest to us here, was the following:

\begin{quote}
\emph{How close can irrational elements of Cantor's set be approximated by rational numbers \\
(i) in Cantor's set, and \\
(ii) by rational numbers not in Cantor's set?} 
\end{quote}

Since the publication of Mahler's note, this question has attracted a huge amount of interest and a wide variety of people have uncovered information about various aspects of this problem. See, for example, \cite{Allen-Barany ref, BroderickFishmanReich2011, Bugeaud2008, BugeaudDurand2016, FishmanMerrillSimmons2018,FishmanSimmons2014_intrinsic,  FishmanSimmons2015_extrinsic, KLW, Kristensen2006, LSV ref, Sch2020, Weiss2001} and references therein. 

Arguably, the first step towards addressing Mahler's problem outlined above was the work of Weiss \cite{Weiss2001}, who showed that $\mu$-almost no point in the Cantor set is \emph{very well approximable}. Here $\mu$ denotes the natural measure on the middle-third Cantor set as defined in \eqref{Cantor measure definition} below. We shall write Cantor set and middle-third Cantor set interchangeably throughout. Recall that Dirichlet's Approximation Theorem tells us that for any $x \in \R$, we have 
\begin{align} \label{Dirichlet}
\left|x-\frac{p}{q}\right|<\frac{1}{q^2}
\end{align}
for infinitely many pairs $(p,q) \in \Z \times \N$. We say that $x \in \R$ is \emph{very well approximable} if the exponent~2 in the denominator of the right-hand side of \eqref{Dirichlet} can be improved, i.e. increased. That is, $x \in \R$ is very well approximable if there exists $\varepsilon > 0$ such that
\begin{align} \label{VWA}
\left|x-\frac{p}{q}\right|<\frac{1}{q^{2+\varepsilon}}
\end{align}
for infinitely many pairs $(p,q) \in \Z \times \N$.

In spite of the work of Weiss \cite{Weiss2001}, in \cite{LSV ref} Levesley, Salp, and Velani were able to prove that there do in fact exist very well approximable numbers (aside from Liouville numbers) in the middle-third Cantor set, thus solving a more precise problem attributed to Mahler which is stated in \mbox{\cite[Problem 35]{Bugeaud - Approximation by Algebraic Numbers}.}

This result prompts further study of Mahler's question from the point of view of \emph{irrationality exponents}. The \emph{irrationality exponent} $\xi(x)$ of $x \in \R$ is defined as: 
\[\xi(x):=\sup\left\{\xi \in \R: \left|x-\frac{p}{q}\right|<\frac{1}{q^{\xi}} \text{ for infinitely many } (p,q) \in \Z \times \N\right\}.\]
By Dirichlet's Approximation Theorem, we know that $\xi(x)\geq 2$ for all $x \in \R$. A real number $x \in \R$ is very well approximable if we have strict inequality here, i.e. if $\xi(x)>2$. The result of Weiss \cite{Weiss2001} shows us that for $\mu$-almost all points in the middle-third Cantor set this is not the case. Nevertheless, the work of Levesley, Salp, and Velani \cite{LSV ref} shows that the set of non-Liouville points in the middle-third Cantor set with irrationality exponent strictly greater than 2 is non-empty. Further results concerning the irrationality exponents of points in the middle-third Cantor set have been obtained in \cite{Bugeaud2008,BugeaudDurand2016}.

 
In this paper, we will be concerned with another aspect of Mahler's problem. In particular, we will be concerned with the problem of how well points in the middle-third Cantor set can be approximated by dyadic rationals; that is, rationals with denominators which are powers of 2. To some extent, our present work is motivated by the work of Levesley, Salp, and Velani \cite{LSV ref} who, in resolving \cite[Problem 35]{Bugeaud - Approximation by Algebraic Numbers},  considered \emph{triadic} approximation in the middle-third Cantor set. In particular, they proved a ``zero-full'' dichotomy for the $\mu$-measure of the points in the middle-third Cantor set which are $\psi$-well approximable by triadic rationals; that is, rationals with denominators which are powers of 3. In fact, the work of Levesley, Salp, and Velani is far more general and they actually proved such a ``zero-full'' dichotomy for the Hausdorff measures of the set in question with respect to general gauge functions. However, in this paper we will only be concerned with $\mu$-measure.

Throughout, we will let $K$ denote the middle-third Cantor set. Recall that $K$ consists of the points $x \in [0,1]$ which have a ternary expansion consisting entirely of 0's and 2's. We will also denote by $\gamma$ the Hausdorff dimension of $K$, namely
\begin{equation} \label{GamDef}
\gamma:= \dimh{K} = \frac{\log{2}}{\log{3}}.
\end{equation}
Throughout, for a subset $X \subset \R$, we will denote the Hausdorff dimension of $X$ by $\dimh{X}$. For a real number $s>0$, we shall write $\cH^s(X)$ to denote the Hausdorff $s$-measure of $X$.
 
We will denote by $\mu$ the natural measure on $K$. More precisely, $\mu$ is the Hausdorff $\gamma$-measure restricted to the middle-third Cantor set, i.e. for $X \subset \R$ which is Borel, 
\begin{align} \label{Cantor measure definition}
\mu(X) = \frac{\cH^{\gamma}(X \cap K)}{\cH^{\gamma}(K)} = \cH^{\gamma}(X \cap K),
\end{align}
where we have the last inequality since it happens that $\cH^{\gamma}(K)=1$. Moreover, $\mu(K)=1$ and so $\mu$ is a probability measure on $K$. For further information on Hausdorff measure and dimension, we refer the reader to \cite{Falconer}.

Given $b \in \N$, we will write
\[
\cA(b)=\{b^n: n=0,1,2,3,\dots\}.
\]
Given $\psi: \N \to \R^+$, let
\[
W_{3}(\psi):=\left\{x \in [0,1]: \left|x-\frac{p}{q}\right| < \frac{\psi(q)}{q} \text{ for infinitely many } (p,q) \in \Z \times \cA(3)\right\}.
\] 
Here $\R^+:=[0,\infty)$.

Regarding triadic approximation in the middle-third Cantor set, the following statement follows immediately from \cite[Theorem 1]{LSV ref}.

\begin{theorem}[Levesley--Salp--Velani \cite{LSV ref}] \label{ThmLSV}
For $\psi: \N \to \R^+$,
\[
\mu(W_{3}(\psi))=
\begin{cases}
0 &\text{if}\quad\sum_{n=1}^{\infty}{\psi(3^n)^{\gamma}}<\infty,\\
&\\
1 &\text{if}\quad\sum_{n=1}^{\infty}{\psi(3^n)^{\gamma}}=\infty.
\end{cases}
\]
\end{theorem}

In this paper, we will be concerned with analogous statements for approximating points in the Cantor set by \emph{dyadic} rationals rather than \emph{triadic} rationals. With this in mind, given a function $\psi: \N \to \R^+$, we will be concerned with

\[W_{2}(\psi):=\left\{x \in [0,1]: \left|x-\frac{p}{q}\right|<\frac{\psi(q)}{q} \text{ for infinitely many } (p,q) \in \Z \times \cA(2)\right\}.\]  

Along the same lines as Theorem \ref{ThmLSV}, we have the following conjecture\footnote{We thank Sanju Velani for suggesting to us that this statement should be true. The present work supports this.}, which represents the clean-cut dichotomy which we expect to be the eventual truth regarding dyadic approximation in the middle-third Cantor set.

\begin{conjecture}[Velani] \label{main conjecture}
For monotonic $\psi: \N \to \R^+$, we have
\[
\mu(W_2(\psi))=
\begin{cases}
0&\text{if}\quad\sum_{n=1}^{\infty}{\psi(2^n)}<\infty,\\
&\\
1 &\text{if}\quad\sum_{n=1}^{\infty}{\psi(2^n)}=\infty.
\end{cases}
\]
\end{conjecture}

\begin{remark} It is possible that Conjecture \ref{main conjecture} may even hold without the monotonicity hypothesis.
\end{remark}

To give a heuristic idea for why one might believe the above conjecture to be correct, suppose for a moment that dyadic rationals of denominator $2^n$ were uniform random variables in $[0,1]$. Then the expected number of dyadic rationals of denominator $2^n$ lying in a subinterval $I \subset [0,1]$ would be $\approx 2^n \times (\text{the length of $I$})$. 

For each $n \in \N$, let $A_n := \bigcup_{a=0}^{2^n}{B\left(\frac{a}{2^n},\frac{\psi(2^n)}{2^n}\right)}$. Then $W_2(\psi) = \bigcap_{j=0}^{\infty}{\bigcup_{n=j}^{\infty}{A_n}}$, and $\bigcup_{n=j}^{\infty}{A_n}$ is a cover of $W_2(\psi)$ for each $j \in \N$. Now, consider a fixed $n \in \N$ and suppose for this $n$ that $\frac{\psi(2^n)}{2^n} \approx 3^{-N}$. In particular, by our assumption on the distribution of dyadic rationals, we would expect $\ll\left(\frac{2^N}{3^N}\right) \times 2^n$ dyadic rationals to lie within distance $\frac{\psi(2^n)}{2^n}$ of the $N$th level of the construction (this is called $K_N$ later on) of the middle-third Cantor set. Thus, this number represents the maximum number of individual balls in $A_n$ which can possibly intersect the middle-third Cantor set. Moreover, it is known that 
\begin{equation} \label{UpperMeasure}
\mu(B(z,r)) \ll r^\gam \qquad (z \in \bR, \quad 0 < r \le 1),
\end{equation}
see for instance \cite{Weiss2001} or \cite[\S2]{Vee1999}. So, we have
\begin{align*}
\mu(A_n) \ll  \left(\frac{2^N}{3^N}\right) \times 2^n \times \left(\frac{\psi(2^n)}{2^n}\right)^{\gamma} 
         \approx 2^N \times \left(\frac{\psi(2^n)}{2^n}\right) \times 2^n \times (3^{-N})^{\gamma}
         = \psi(2^n).
\end{align*}
Combining the above heuristics with the convergence Borel--Cantelli Lemma (Lemma \ref{first Borel-Cantelli}) gives rise to the convergence part of the above conjecture. 

By considering balls of radius $\psi(2^n)/2^n$ centred at triadic rationals within the Cantor set, we can obtain a similar heuristic for the complementary lower bound $\mu(A_n) \gg \psi(2^n)$. If we knew that $\mu(A_n \cap A_m)$ were $O(\mu(A_n) \mu(A_m))$ in some suitably averaged sense, then the divergence Borel--Cantelli Lemma \cite[Proposition 2]{BDV ref} would complete the proof. 

Unfortunately, we cannot prove such bold claims about the interaction of the dyadic rationals with the middle-third Cantor set. Unlike in the case of triadic rationals, where we know exactly how they are distributed with respect to the middle-third Cantor set $K$, we know very little about how dyadic rationals are distributed with respect to $K$. For this reason, studying dyadic approximation in the Cantor set is significantly harder than studying triadic approximation in the Cantor set and, as of yet, we have been unable to prove Conjecture \ref{main conjecture} in full. As a first step towards this conjecture, the following convergence statement is relatively straightforward to establish --- we will provide a direct proof in Section~\ref{basic result section}.

\vbox{
\begin{proposition} \label{basic proposition}
For $\psi: \N \to \R^+$, if
\begin{equation} \label{BenchmarkSeries}
\sum_{n=1}^{\infty}{\psi(2^n)^{\gamma}}<\infty,
\end{equation}
then $\mu(W_{2}(\psi))=0$.
\end{proposition}}

Using Fourier analysis, we are able to establish the following improvement upon the above benchmark result.

\begin{theorem}[Main Convergence Theorem] \label{main convergence theorem}
If 
\begin{align} \label{main convergence theorem sum} 
\sum_{n=1}^{\infty} (2^{-\log n / (\log \log n \cdot \log \log \log n)} \psi(2^n)^\gamma+\psi(2^n)) &< \infty,
\end{align} 
then $\mu(W_2(\psi)) = 0$.
\end{theorem}

\begin{remark} \label{ConvergenceRemark} We may deduce Proposition \ref{basic proposition} from Theorem \ref{main convergence theorem} via two observations. First, by replacing $\psi$ by $\psi_0: y \mapsto \min \{3/4, \psi(y) \}$, we may assume that $\psi(2^n) < 1$ for all $n$, since
\[
\| 2^n \alp \| < \psi(2^n) \quad \Longleftrightarrow \quad \| 2^n \alp \| < \psi_0(2^n).
\]
Second, observe that
\[
2 \psi(2^n)^\gam > 2^{-\log n / (\log \log n \cdot \log \log \log n)} \psi(2^n)^\gamma+\psi(2^n).
\]
Thus, assuming \eqref{BenchmarkSeries}, it follows that \eqref{main convergence theorem sum} also holds and we may apply Theorem \ref{main convergence theorem} to deduce Proposition \ref{basic proposition}. The improvement is super-logarithmic assuming $\sum_{n=1}^\infty \psi(2^n) < \infty$. For instance, \eqref{main convergence theorem sum} holds for $\psi(2^n) = (\log n)^{\alpha} n^{-1/\gam}$
whenever $\alpha \in \mathbb R$, whereas \eqref{BenchmarkSeries} fails even for $\psi(2^n) = n^{-1/\gam}$. 
\end{remark}

From Theorem \ref{ThmLSV}, we know that for $\mu$-almost every $\alp\in K$ the inequality
\begin{align} \label{special case triadic}
\| 3^n \alp \| &< n^{-\log 3 / \log 2}
\end{align}
admits infinitely many solutions $n \in \bN$, where $\|x\|$ denotes the distance from $x$ to the nearest integer for $x \in \R$. However, we see from Theorem \ref{main convergence theorem} that for $\mu$-almost every $\alp$ the inequality
\begin{equation} \label{SpecialCase}
\| 2^n \alp \| < n^{-\log 3 / \log 2}
\end{equation}
has at most finitely many solutions. Thus, the behaviour is very different in the case of dyadic approximation in the Cantor set. Observe that it is not possible to obtain this conclusion by only using Proposition \ref{basic proposition}, as has already been indicated in Remark \ref{ConvergenceRemark}. In this sense, the additional super-logarithmic decaying factor
	\[
	2^{-\log n / (\log \log n \cdot \log \log \log n)}
	\] 
	marks a significant difference. 
	
The difference in behaviour of dyadic and triadic approximation here is further emphasised by the following theorem due to Bugeaud \cite[Theorem 7.17]{Bugeaud distribution mod 1 ref} which, despite the differing metric statements, asserts that for each of the equations \eqref{special case triadic} and \eqref{SpecialCase} there exist uncountably many real numbers in the middle-third Cantor set for which these equations are satisfied only finitely often.

\begin{theorem}[Bugeaud \cite{Bugeaud distribution mod 1 ref}] \label{Bugeaud 7.17}
There exists a positive real number $c$, and uncountably many real numbers $x \in K$ which are badly approximable and, for all integers $b \geq 2$ and $n \geq 1$, satisfy
\[\|b^n x\| > b^{-cb(\log{b})}.\]
\end{theorem}

\noindent Here we ask the following complementary question.

\begin{qn} Does there exist $\alp \in K \setminus \bQ$ such that \eqref{SpecialCase} has infinitely many solutions?
\end{qn}

Complementing Theorem \ref{main convergence theorem}, we prove the following statement towards the divergence part of Conjecture \ref{main conjecture}.

\begin{theorem}[Main Divergence Theorem] \label{main divergence theorem}
For $\psi(2^n) = 2^{- \log \log n / \log \log \log n}$, we have \mbox{$\mu(W_2(\psi)) = 1$.}
\end{theorem}

In Section \ref{conditional results section}, we also provide some additional conditional results and empirical evidence which provide further support in favour of Conjecture \ref{main conjecture}. We will see in particular that a modest refinement of a key estimate, namely \eqref{binary ternary digit changes inequality} below, would lead to a sharp convergence theory. Moreover, consider the special case 
\[
\varphi_a: 2^n \mapsto n^{-a},
\]
for $a \ge 0$. It is clear that
\[
\mu(W_2(\varphi_0)) = 1,
\]
and we know from Theorem \ref{main convergence theorem} that whenever $a \geq 1/\gamma$ we have
\[
\mu(W_2(\varphi_a)) = 0.
\]
We shall see conditionally, in Theorems \ref{conditional convergence theorem} and \ref{conditional divergence theorem}(\ref{a}), that
\[
\mu(W_2(\varphi_a)) = \begin{cases}
0, &\text{if } a > 1 \\
1, &\text{if } a \le 1.
\end{cases}
\]

\subsection{Main ideas behind the proofs and some preliminaries}

The key to proving results of the flavour we are considering is an understanding of how dyadic rationals are distributed with respect to the middle-third Cantor set. This is not an easy task, and is a variant of the ``times two, times three'' phenomenon (see, for example, \cite{Furstenberg67, Furstenberg}). Nevertheless, a fairly straightforward counting argument enables us to deduce enough information to establish Proposition \ref{basic proposition}. Refining the argument using Fourier analysis, we are also able to establish Theorems \ref{main convergence theorem} and \ref{main divergence theorem}. The relevance of counting dyadic rationals near the middle-third Cantor set will be described in greater detail in Section \ref{dyadics section}.

In Section \ref{dyadics section}, we see via Fourier analysis that the problem at hand is intimately connected to the number of binary and ternary ``digit changes'' in numbers of the form $2^{n}m$, where $n, m \in \N$. For $b \in \Z$ such that $b \ge 2$ and $y\in\mathbb{R},$ let $D_{b}(y)$ denote the number of digit changes in the $b$-adic expansion of $y$; that is, $D_{b}(y)$ denotes the number of consecutive pairs of distinct digits in the $b$-adic expansion of $y$. We want to understand the sum $D_2(y)+D_3(y).$ To this end, we will make use of the following inequality, originally due to Stewart \cite{Stewart ref} and extended by Bugeaud, Cipu, and Mignotte \cite{Bugeaud-Cipu-Mignotte}. Integers $a$ and $b$ are \emph{multiplicatively independent} if, for any $m, n \in \bZ$, we have that $a^m=b^n$ implies that $m=n=0$. The statement as written below can be found in \cite[Theorem~6.9]{Bugeaud distribution mod 1 ref}. 

\begin{theorem}[Stewart \cite{Stewart ref}, Bugeaud--Cipu--Mignotte \cite{Bugeaud-Cipu-Mignotte}] \label{digit changes theorem}
Let $a$ and $b$ be multiplicatively independent integers. Then, there exists an effectively computable integer $c$, which depends only on $a$ and $b$, such that for every natural number $n \geq 20$, we have 
\begin{align*}
D_a(n)+D_b(n) \geq \frac{\log{\log{n}}}{\log{\log{\log{n}}}+c}-1.
\end{align*}
\end{theorem}

\noindent This is a deep result which uses Baker's work on linear forms in logarithms \cite{Baker}. The following more specialised statement follows easily from Theorem \ref{digit changes theorem}. 

\begin{lemma} \label{digit changes lemma}
For sufficiently large $n \in \N$, we have 
\begin{align} \label{binary ternary digit changes inequality}
D_{2}(n)+D_{3}(n) \gg \frac{\log{\log{n}}}{\log{\log{\log{n}}}},
\end{align} 
where the implicit constant is absolute.
\end{lemma} 

\noindent \textbf{Notation.} We use the Bachmann--Landau and Vinogradov notations throughout: for functions $f$ and positive-valued functions $g$, we write $f \ll g$, or $g \gg f$, or $f = O(g)$, if there exists a constant $C>0$ such that $|f(x)| \le C g(x)$ for all $x$.

In addition to our proofs relying heavily on Fourier analysis and the bounds on base 2 and base~3 digit changes discussed above, we will make use of both the standard convergence Borel--Cantelli Lemma and the Chung--Erd\H{o}s inequality from probability.

\begin{lemma}[Convergence Borel--Cantelli Lemma] \label{first Borel-Cantelli}
Let $(\Omega,\mathcal{A},m)$ be a finite measure space and let $\{E_n\}_{n\geq 1} \subset  \mathcal{A}$ be a sequence of $m$-measurable sets in $\Omega$. 
If
\[\sum_{n=1}^{\infty}{m(E_n)} < \infty,\]
then
\[m(\limsup_{n \to \infty}{E_n}) = 0.\]
\end{lemma}

For the proof of the divergence result Theorem \ref{main divergence theorem}, we use the  following inequality established by Chung and Erd\H{o}s in \cite[Equation~4]{Chung-Erdos ref}. This is similar in flavour to the divergence counterpart of the Borel--Cantelli Lemma (see e.g. \cite[Proposition 2]{BDV ref}).

\begin{lemma}[Chung--Erd\H{o}s Inequality] \label{Chung-Erdos}
Let $N\geq 1$ be an integer. Let $(\Omega, \mathcal{A}, m)$ be a probability space, and let  $\{E_n\}_{n=1}^{N} \subset  \mathcal{A}$ be an arbitrary sequence of $m$-measurable sets in $\Omega$. Then, if $m\left(\bigcup_{n=1}^{ N}{E_n}\right)>0$, we have 
\[m\left(\bigcup_{n=1}^{N}{E_n}\right) \geq \frac{\left(\sum_{s=1}^{N}{m(E_s)}\right)^2}{\sum_{s,t=1}^{N}{m(E_s \cap E_t)}}.\]
\end{lemma}

In order to prove  the full measure statements given by Theorem \ref{main divergence theorem} and later by Theorem~\ref{conditional divergence theorem}, we first intersect our sets of interest with an arbitrary interval $I$ centred in the Cantor set and show that these intersections have measure proportional to the measure of the interval $I$. Once this has been achieved, the following lemma yields the full measure statements we ultimately desire since $\mu$ is doubling (see Section \ref{FourierAnalysis} for a definition). Lemma \ref{positive measure lemma}, as stated below, can be found in \cite[Proposition 1]{BDV ref}. We also refer the reader to \cite{BDV ref} for several other applications of Lemma \ref{positive measure lemma}. 

\begin{lemma} \label{positive measure lemma}
Let $(\Omega, d)$ be a metric space, and let $m$ be a finite doubling measure on $\Omega$ such that any open set is measurable. Let $E$ be a Borel subset of $\Omega$. Assume that there are constants $r_0, c > 0$ such that for any ball $B$ of radius $r(B) < r_0$ and centre in $\Omega$ we have that 
\[m(E \cap B) \geq c m(B).\]
Then $E$ has full measure in $\Omega$, i.e. $m(\Omega \setminus E)=0$.
\end{lemma}

\section{Dyadic rationals near the Cantor set} \label{dyadics section}

We recall a natural construction of the middle-third Cantor set. Set $K_0:=[0,1]$ and remove the open middle-third of this interval to obtain $K_1:=[0,\frac{1}{3}]\cup[\frac{2}{3},1]$. Subsequently, remove the open middle-third of each of these component intervals to obtain $K_2$ which will consist of 4 subintervals of length $3^{-2}$, and repeat this process so that $K_N$ consists of the $2^N$ intervals of length $3^{-N}$ which remain after removing the open middle-thirds from each of the component subintervals of $K_{N-1}$. Then $K = \bigcap_{N=0}^{\infty}{K_N}.$

For $N \in \bN$, let $\cL_N$ be the set of rational numbers of the form $\frac{b}{3^N}$ where $b$ is an integer such that $0 \leq b \leq 3^N$ and the ternary expansion of $b$ only contains 0's and 2's. Thus, $\cL_N$ corresponds to the left-most endpoints of the intervals comprising $K_N$. We also define $\cR_N:=\{x=1-y: y \in \cL_N\}$ to be the set of right-most endpoints of intervals in $K_N$. We write $\cC_N := \cL_N \cup \cR_N$ to denote the set of all endpoints of intervals comprising $K_N$. Furthermore, note that $\cC_N$ is the set of rationals in $[0,1]$ with denominator $3^N$ which lie in the Cantor set. For $n, M \in \bN$, let
\begin{align*}
A_n := \bigcup_{a = 0}^{2^n} B\left(\frac{a}{2^n}, \frac{\psi(2^n)}{2^n}\right), \quad \text{and} \quad
A_n(M) := \bigcup_{a = 0}^{2^n} B\left(\frac{a}{2^n}, \frac{1}{3^{M}}\right). 
\end{align*}

Observe that if $N \in \N$ and $x \in \cC_N$ then $\mu \left(B\left(x, \frac{3^{-N}}{2}\right)\right)  = 2^{-(N+1)}$. Indeed, each component interval in $K_{N}$ has $\mu$-measure $2^{-N}$, and $B\left(x, \frac{3^{-N}}{2}\right)$ contains half of such an interval and intersects no others. Moreover, these balls are disjoint. Similarly, for any $x \in \bR$ we have \mbox{$\mu (B(x, 3^{-N})) \le 2^{-(N-1)}$},  since $B(x, 3^{-N})$ intersects at most two of the component intervals in $K_{N}$.

As alluded to earlier, in several of the lemmas below, we will intersect our sets of interest with an interval $I = B(z,r)$, where $z \in K$ and $r > 0$. This generality will be needed in the divergence theory in order to obtain full-measure results, as opposed to just positive-measure results, and is standard in metric Diophantine approximation (for several examples, see \cite{BDV ref}). The idea is that dyadic rationals should be evenly distributed in $K = \supp(\mu)$. For such an interval $I=B(z,r)$ or $I=[0,1]$, for $n \in \bN$ and $t \in (0, \infty)$, we write
\[
t A_n = \bigcup_{a = 0}^{2^n} B\left(\frac{a}{2^n}, t\frac{\psi(2^n)}{2^n}\right), \quad
tA_n(M) := \bigcup_{a = 0}^{2^n} B\left(\frac{a}{2^n}, \frac{t}{3^{M}}\right), \quad \text{and} \quad tI = B(z, tr).
\]
We adopt analogous notation for translates of $A_n$ and for $A_n(M)$, i.e. for $n \in \N$ and $\tet \in \R$,
\[A_n + \tet = \bigcup_{a = 0}^{2^n} B\left(\frac{a}{2^n}+\tet,\frac{\psi(2^n)}{2^n}\right) \quad \text{and} \quad A_n(M)+\tet = \bigcup_{a = 0}^{2^n} B\left(\frac{a}{2^n}+\tet, \frac{1}{3^{M}}\right).\]
Finally, we introduce the hybrid notations
\[
tA_n + \tet = \bigcup_{a = 0}^{2^n} B\left(\frac{a}{2^n}+\tet, t \frac{\psi(2^n)}{2^n}\right) \quad \text{and} \quad tA_n(M)+\tet = \bigcup_{a = 0}^{2^n} B\left(\frac{a}{2^n}+\tet, \frac{t}{3^{M}}\right),\]
for $n \in \bN$, $t \in (0, \infty)$, and $\tet \in \bR$.

We will assume throughout this section that $\psi(2^n) < 1$ for all $n \in \N$. This is not at all restrictive for the purposes of proving Proposition \ref{basic proposition}, Theorem \ref{main convergence theorem}, or Theorem \ref{main divergence theorem}. Indeed, for the first two of these statements we may replace $\psi$ by $y \mapsto \min \{ 3/4, \psi(y) \}$ as in Remark \ref{ConvergenceRemark}, and in Theorem~\ref{main divergence theorem} the condition is clearly met.

The $\mu$-measure of a union of balls can be estimated by counting nearby triadic rationals in $C_{N}$ for a sufficiently large $N \in \N$. This is formalised below.

\begin{lemma} \label{connection} Let $I = B(z,r)$, where $z \in K$ and $r>0$, \textbf{or} $I = [0,1]$. Let $n_0(I,\psi) \in \N$ be sufficiently large so that if $n \geq n_0(I, \psi)$ then $2^{-n}<r$. Let $n \ge n_0(I,\psi)$ and $N$ be positive integers such that $3^{-N} \le \frac{\psi(2^n)}{5 \cdot 2^n}$, and let $B_n$ be a translate of $A_n$. Then
\[
2^{-(N+1)} | \cC_N \cap 0.2 B_n \cap 0.2 I| \le \mu(B_n \cap I) \le 2^{-(N-1)} |\cC_N \cap 5 B_n \cap 5 I|.
\]
\end{lemma}

\begin{proof} If $x \in \cC_N \cap 0.2B_n \cap 0.2I$ then $B(x, 3^{-N}/2) \subseteq B_n \cap I$. That $B(x,3^{-N}/2) \subseteq B_n$ is clear upon noting our choice of $N$ and recalling that $B_n$ consists of disjoint intervals of length $2\times\frac{\psi(2^n)}{2^n}$. The inclusion $B(x,3^{-N}/2) \subseteq I$ follows by the triangle inequality and our choices of $N$ and $n_0(I, \psi)$. Hence, by disjointness (since any two distinct points in $\cC_N$ are distance at least $3^{-N}$ apart), we have
\[
\mu(B_n \cap I) \ge \sum_{x \in \cC_N \cap 0.2B_n \cap 0.2I} \mu(B(x, 3^{-N}/2)) = 2^{-(N+1)} | \cC_N \cap 0.2 B_n \cap 0.2I| .
\]

For the second inequality, we cover $[0,1]$ by $\{B(x, 3^{-N}): x \in S_{N}\}$, where $S_{N}$ is the set of rationals with denominator $3^N$ in $[0,1]$. For each $x \in S_N$ we have
\[
\mu(B_n \cap I \cap B(x,  3^{-N})) \le \mu(B(x,  3^{-N})) \le2^{-(N-1)}.
\]
If $x \in S_N \setminus \cC_N$ then $B(x, 3^{-N})$ does not intersect the Cantor set, so
\[
\mu(B_n \cap I \cap B(x,  3^{-N})) \le \mu(B(x,  3^{-N})) = 0.
\]
Finally, if $x \in S_N \setminus (5B_n \cap 5I)$ then $B(x, 3^{-N})$ does not intersect $B_n \cap I$, so
\[
\mu(B_n \cap I \cap B(x,  3^{-N})) = 0.
\]
To see this, one can argue by contradiction, again employing the triangle inequality and recalling our choices of $N$ and $n_0(I, \psi)$.

Putting everything together yields
\begin{align*}
\mu(B_n \cap I) \le \sum_{x \in S_N} \mu(B_n \cap I \cap B(x,  3^{-N})) \le 2^{-(N-1)} |\cC_N \cap 5B_n \cap 5I|,
\end{align*}
as required.
\end{proof}

We also require the following inequality, which enables us to pass between Cantor levels $K_N$. Here $N$ is roughly the level at which $A_n$ lives since, up to a multiplicative constant, the component intervals in $K_N$ are the same length as the intervals comprising $A_n$.

\begin{lemma} \label{DroppingDown} Let $n, M, N \in \bN$ be such that
\begin{equation} \label{ValidDrop}
\frac{\psi(2^n)}{125 \cdot 2^n} \le 3^{-N} \le \frac{5 \psi(2^n)}{2^n} \leq 3^{-M} \leq 2^{9-n}.
\end{equation}
Then
\[
|\cC_N \cap 5A_n| \ll |\cC_M \cap 5A_n(M)|,
\]
where the implied constant is absolute.
\end{lemma}

\begin{proof} 
Suppose $x \in \cL_N \cap 5A_n$; the corresponding bound for $x \in \cR_N$ follows by symmetry. Since $x \in \cL_N$, we have $x=\frac{a}{3^N}$ for some integer $a\in [0, 3^N)$, where the ternary expansion of $a$ consists only of 0's and 2's. As $x \in 5A_n$, there exists an integer $b \in [0, 2^n]$ such that $\left|x-\frac{b}{2^n}\right|<\frac{5 \psi(2^n)}{2^n}$. Thus, $|\cL_N \cap 5A_n|$ is bounded above by the number of integer solutions $(a,b)$ to the inequality
\[\left|\frac{a}{3^N}-\frac{b}{2^n}\right|<\frac{5\psi(2^n)}{2^n},\]
where $0 \leq b \leq 2^n$, $0 \leq a < 3^N$, and all of the ternary digits of $a$ are 0 or 2.

Let us now decompose $a$ by writing $a = a_1a_2$, where $a_1$ represents the first $M$ ternary digits of $a$ and $a_2$ represents the remaining $N-M$ digits. From this, we see that $|\cL_N \cap 5A_n|$ is bounded above by the number of integer solutions $(a_1,a_2,b)$
to
\begin{equation} \label{higher}
\Bigl| \frac { 3^{N-M} a_1 + a_2} {3^N} - \frac{b}{2^n} \Bigr| < \frac{5 \psi(2^n)}{2^n},
\end{equation}
where $0 \le a_1 < 3^M$, $0 \le a_2 < 3^{N-M}$, $0 \le b \le 2^n$, and the ternary digits of $a_1, a_2$ are all 0 or 2. Next, we note that 
\begin{equation*}
\frac{a_1}{3^M} \in 5A_n(M),
\end{equation*}
since
\begin{equation} \label{tri}
\left|\frac{a_1}{3^M}-\frac{b}{2^n}\right| \leq \left|\frac{a_1}{3^M}+\frac{a_2}{3^N}-\frac{b}{2^n}\right|+\frac{a_2}{3^N} < \frac{5\psi(2^n)}{2^n} + 3^{-M} \leq \frac{2}{3^M}.
\end{equation}

Given $a_1$, we see from \eqref{tri} that $b/2^n$ is forced to lie in an interval of length $\frac{4}{3^M}$ centred at~$\frac{a_1}{3^M}$. By hypothesis, $\frac{4}{3^M} \ll 2^{-n}$, so there are $O(1)$ possibilities for $b$. Next, suppose we are given $a_1$ and~$b$. Then, by \eqref{higher}, $a_2$ must lie in an interval of length $\frac{3^N \times 10 \times \psi(2^n)}{2^n}$ centred at $\frac{3^N b}{2^n}-\frac{3^{N}a_1}{3^M}$. Thus, there are $O(1)$ solutions $a_2$ to \eqref{higher}, by our assumption that $3^{-N} \asymp \frac{\psi(2^n)}{2^n}$. Finally, since $\frac{a_1}{3^M} \in \cL_M \cap 5A_n(M)$, we conclude that there are $O(|\cL_M \cap 5A_n(M)|)$ solutions to \eqref{higher} in total. By symmetry we have $|\cR_N \cap 5A_n| = |\cL_N \cap 5A_n|$, so
\[
|\cC_N \cap 5A_n| = 2 |\cL_N \cap 5A_n| \ll |\cC_M \cap 5 A_n(M)|. \qedhere
\] 
\end{proof}

\subsection{Proof of the basic result} \label{basic result section}

We presently provide a direct proof of Proposition \ref{basic proposition}. Recall that we assume that $\psi(2^n)<1$ for all $n \in \N$. If $n \ge 7$ is an integer, and $N \in \bN$ is such that
\[
3^{-N} \le \frac{\psi(2^n)}{5 \cdot 2^n} < 3^{-(N-1)}
\]
then, by Lemma \ref{connection} (taking $I=[0,1]$), we have
\[
\mu(A_n) \le 2^{-(N-1)} |\cC_N \cap 5A_n|.
\]
Choose $M \in \bN$ such that
\[
3^{-M} < 2^{5-n} \le  3^{-(M-1)},
\]
and note that $M \le N$. We must also have
\[
|\cC_N \cap 5A_n| \ll |\cC_M \cap 5A_n(M)|,
\]
by Lemma \ref{DroppingDown}. Therefore, since $3^{-N} \asymp \frac{\psi(2^n)}{2^n}$, $3^{-M} \asymp 2^{-n}$, and $|\cC_M \cap 5A_n(M)|$ is trivially bounded above by $|\cC_M|=2^{M+1}$, we have
\[
\mu(A_n) \ll 2^{M-N} = (3^{M-N})^{\gamma} \ll (2^n / 3^N)^\gam \ll \psi(2^n)^\gam
\]
for $n \ge 7$. Hence
\[
\sum_{n=1}^{\infty} \mu(A_n) \ll \sum_{n=1}^{\infty} \psi(2^n)^\gam
\]
converges, and the convergence Borel--Cantelli Lemma (Lemma \ref{first Borel-Cantelli}) completes the proof.

\subsection{Fourier analysis} \label{FourierAnalysis}

In this section, we use Fourier analysis to count instances of dyadic rationals being close to triadic rationals in the Cantor set. Recall that $\mu$ possesses the following two properties:
\begin{enumerate} 
\item (Positivity) If $z \in K$ and $r > 0$, then $\mu(B(z,r)) > 0$, and
\item (Doubling) If $z \in K$ and $r > 0$, then $\mu(B(z,2r)) \ll \mu(B(z,r))$.
\end{enumerate}
In fact $\mu$ is Ahlfors--David regular; that is, for $z \in K$ and $r \in (0,1]$, we have
\[
\mu(B(z,r)) \asymp r^\gam,
\]
see \cite[\S 6.1]{LSV ref} or \cite{Falconer}. This property implies the properties (1) and (2) stated above.

Before commencing in earnest, we estimate the measure of an interval in a discrete fashion. This is a simpler analogue of Lemma \ref{connection}.

\begin{lemma} \label{MeasureOfInterval} Let $I = B(z,r)$, where $z \in K$ and $0<r<1$, and let $L\ge L_0(I)$ be a large positive integer. In particular, we require $L_0(I)$ to be sufficiently large that $\frac{3^{-L_0(I)}}{2}<\frac{r}{5}$. Then
\[
\mu(I) \asymp \frac{ |\cC_L \cap I| }{2^L}.
\]
The implicit constants are absolute. Moreover, if $B = B(z,r)$, where $z \in \bR$ and $0 <r < 1$, and $L$ is a positive integer for which $3^{1-L} < r$, then
\begin{equation} \label{VersatileMeasure}
\mu(B) \le 2^{-(L-1)} |\cC_L \cap 5B|.
\end{equation}
\end{lemma}

\begin{proof} By the doubling property of $\mu$, it suffices to show that
\begin{align} \label{doubling}
2^{-(L+1)} |\cC_L \cap 0.2 I| &\le \mu(I) \le 2^{-(L-1)} |\cC_L \cap 5I|.
\end{align}
To see that this is indeed sufficient, suppose first that we wish to show that $2^L \mu(I) \ll |\cC_{L} \cap I|$. This would follow from $2^L \mu(0.2I) \leq 2|\cC_{L} \cap I|$, since doubling gives 
\[\mu(I) \le \mu(1.6I) \ll \mu(0.8I) \ll \mu(0.4I) \ll \mu(0.2I).\]
That is, it would suffice to show that $2^L \mu(J) \leq 2|\cC_{L} \cap 5J|$ for $J=0.2I$. By a similar argument, the inequality $2^{L}\mu(I) \gg |\cC_{L} \cap I|$ would follow if we could show that \mbox{$2^{L+1}\mu(J') \geq |\cC_{L} \cap 0.2J'|$} for $J'=5I$, since $\mu(5I) \ll \mu(I)$ by the doubling property. Since $I$, and therefore also $J$ and $J'$, are arbitrary intervals, it suffices to show~\eqref{doubling}.

If $x \in \cC_L \cap 0.2I$ then $B(x, 3^{-L}/2) \subseteq I$, by the triangle inequality. Hence, by disjointness (since distinct points in $\cC_L$ are distance at least $3^{-L}$ apart), we have
\[
\mu(I) \ge \sum_{x \in \cC_L \cap 0.2I} \mu \left(  B\left(x, \frac{3^{-L}}{2}\right)  \right)  = 2^{-(L+1)} |\cC_L \cap 0.2I|.
\]

For the upper bound, we cover $[0,1]$ by $\{B(x,3^{-L}): x \in S_{L}\}$, where $S_{L}$ is the set of rationals with denominator $3^L$ in $[0,1]$. For each $x \in S_L$ we have
\[
\mu(I \cap B(x,3^{-L})) \le \mu(B(x,3^{-L})) \le 2^{-(L-1)}.
\]
If $x \in S_L \setminus \cC_L$ then $B(x,3^{-L})$ does not intersect the Cantor set, so
\[
\mu(I \cap B(x,3^{-L})) = 0.
\]
Finally, if $x \in S_L \setminus 5I$ then $B(x,3^{-L})$ does not intersect $I$, so
\[
\mu(I \cap B(x,3^{-L})) = 0.
\]
Therefore
\[
\mu(I) \le \sum_{x \in S_L} \mu(I \cap B(x,3^{-L})) \le 2^{-(L-1)} |\cC_L \cap 5I|,
\]
as required.

For the final statement, witness that in proving \eqref{doubling} we did not use the assumption that $z \in K$.
\end{proof}

\subsubsection{Schwartz functions, bump functions, and tempered distributions}

We briefly discuss some theory in preparation for the Fourier analysis. Full details can be found in the distribution theory textbooks of Mitrea \cite{Mit2018}, and Friedlander and Joshi \cite{FJ1998}. We restrict ourselves to the one-dimensional setting. 

Throughout this paper, we write $e(x) = e^{2 \pi i x}$ for $x \in \bR$. A \emph{Schwartz function} is a function $f: \bR \to \bC$ such that if $a,b \in \Z$ are such that $a,b \geq 0$, then
\[
\sup_{x \in \bR} |x^b f^{(a)}(x)| < \infty,
\]
where $f^{(a)}$ is the $a^{\text{th}}$ derivative of $f$. \emph{Schwartz space}, denoted $\cS(\bR)$, is the complex vector space of Schwartz functions, together with a natural topology \cite[\S 14.1]{Mit2018}. The \emph{Fourier transform} of a Schwartz function $f$ is
\[
\hat f: \bR \to \bC, \qquad \hat f (t) = \int_\bR f(\alp) e(-t \alp) \d \alp,
\]
and this defines an automorphism of Schwartz space \cite[Theorem 3.25]{Mit2018}, with the inverse operator sending $g \in \cS(\bR)$ to
\[
t \mapsto \int_\bR g(\alp) e(t\alp) \d \alp.
\]
Note that the normalisation in \cite{Mit2018} is slightly different. 

A \emph{bump function} is a $C^\infty$, compactly supported function $\phi:\bR \to [0,\infty)$. In the sequel, we fix a smooth bump function $\phi$ supported on $[-2,2]$ such that $\phi(x)=1$ for $x\in [-1,1]$ and $0 \le \phi(x) \le 1$ for all $x \in \bR$. Such a function exists, by \cite[Theorem 1.4.1]{FJ1998}. As explained in \cite[\S3.1]{Mit2018}, bump functions are Schwartz, and have the rapid Fourier decay property that for each number $N>0,$
\begin{equation} \label{RapidDecay}
\hat \phi(t) \ll_N (1 + |t|)^{-N},
\end{equation}
where the subscript $N$ denotes that the implicit constant depends on $N$. We note that the bound in \eqref{RapidDecay} may be deduced from \cite[Eq. (3.1.12)]{Mit2018}.

A \emph{tempered distribution} is a continuous linear functional $\cS(\bR) \to \bC$. The archetypal example of a tempered distribution is the Dirac delta function, which sends $g \in \cS(\bR)$ to $g(0)$, but is not a function in the classical sense. Note that a Schwartz function $f$ induces a tempered distribution $T_f$ sending $g \in \cS(\bR)$ to
\[
\int_\bR f(\alp) g(\alp) \d \alp.
\]
We define \emph{translates} of Schwartz functions $f$ and of tempered distributions $u$ by
\[
\tau_a f (x) = f(x - a) \qquad (a,x \in \bR)
\]
and
\[
\tau_a u (g) = u(\tau_{-a}g) \qquad (a \in \bR, \: g \in \cS(\bR)).
\]
For example, denoting by $\del$ the Dirac delta function, the distribution $\tau_a \del$ sends $g \in \cS(\bR)$ to $g(a)$.

The \emph{Fourier transform} of a tempered distribution $u: \cS(\bR) \to \bC$ is the tempered distribution
\[
\hat u: \cS(\bR) \to \bC, \qquad \hat u (f) = u(\hat f).
\]
Writing $g$ as the Fourier transform of $f \in \cS(\bR)$, we obtain the \emph{Parseval formula}
\begin{align} \label{Parseval}
u(g) &= \hat u (g^\vee),
\end{align}
where $g^\vee$ is the inverse Fourier transform of $g$; that is,
$g^\vee(t) = \hat{g}(-t)$.

The Fourier transform of a translate $\tau_a u$ of a tempered distribution is 
\[
g \mapsto \hat u (ge(-a \cdot)).
\]
Indeed, we compute that if $u$ is a tempered distribution, $a \in \bR$, and $g \in \cS(\bR)$, then
\[
\widehat{\tau_a u}(g) = \tau_a u (\hat g) = u(\tau_{-a} \hat g) = u(x \mapsto \hat g (x+a)) = u(\widehat{g e(-a \cdot)})
= \hat u(g e(-a \cdot)).
\]
To see the fourth equality, observe that the Fourier transform of $t \mapsto g(t) e(-at)$ sends $x$ to
\[
\int_\bR g(t) e(-at) e(-tx) \d t = \int_\bR g(t) e( -t (x+a)) \d t = \hat g(x+a).
\]

A continuous, bounded function $F$ is the \emph{function type} of a distribution $u$ if 
\[
u(f) = \int_\bR f(t) F(t) \d t \qquad (f \in \cS(\bR)).
\]
The \emph{distributional Poisson summation formula} \cite[Theorem 8.5.1]{FJ1998} asserts that if $g \in \cS(\R)$ then
\[
\sum_{n \in \bZ} \tau_n \del(g) =  \sum_{n \in \bZ}   \int_\bR g(t) e(nt) \d t.
\]
This admits the following generalisation.

\begin{lemma} \label{periodicity} Let $u$ be a tempered distribution whose Fourier transform has function type $F$. Then 
\[
\sum_{n \in \bZ} \widehat{\tau_n u} =\sum_{n \in \bZ} \tau_n \del (F \cdot).
\]
\end{lemma}

\begin{proof} If $g \in \cS(\bR)$ then 
\begin{align*}
\sum_{n \in \bZ} \widehat{\tau_n u}(g) 
= \sum_{n \in \bZ} \widehat{\tau_{-n} u}(g)
= \sum_{n \in \bZ} \hat u (ge(n \cdot))  =
\sum_{n \in \bZ} \int_\bR F(t) g(t)  e(nt) \d t 
 = \sum_{n \in \bZ} \tau_n \del (Fg),
\end{align*}
where for the final inequality we have applied the distributional Poisson summation formula to the Schwartz function $Fg$.
\end{proof}

\subsubsection{Counting using Fourier analysis}

We now proceed with the Fourier analysis. 

\begin{lemma}\label{FOURIERCOUNT} Let $y \in \R$, and let $n,k  \in \bN$. Let $I$ be either $[0,1]$ or a subinterval of $[0,1]$ whose endpoints do not lie in $K$. Let
\[
f(\alp) = f_n(\alp; y) = \sum_{b = 0}^{2^n - 1} \phi_r ( \alp + y - b/2^n), \qquad \phi_r(\bet) = \phi( r \bet), \qquad r = 2^{n+k}.
\]
If $I=[0,1]$ then let $L_0(I) = 0$. Otherwise, let $L_0(I)$ be a positive integer such that $3^{-L_0(I)}$ is less than the distance from the boundary of $I$ to the boundary of $K$. Finally, let $M \ge L \ge L_0(I)$ be integers. Then, for $T,N \in \bN$ such that $N > 1$, we have
\begin{align*}
&\sum_{x \in \cL_M \cap I \mmod 1} f(x) \\
&= 2^{M-L-k} \sum_{|m|\leq 2^k T} \hat{\phi}(2^{-k}m)e(2^n m y)\sum_{x \in \cL_L \cap I} e(2^n m x) \prod_{j = L+1}^M \frac{1 + e(2^{n+1}m/3^j)}2\\ &+ O_N\left(2^{M-L}\frac{|\cL_L\cap I|}{T^N}\right),
\end{align*}
where the subscript $N$ indicates that the implied constant depends on $N.$ Here $x \in \cL_M \cap I \mmod 1$ means that $x + m \in \cL_M \cap I$ for some $m \in \bZ$.
\end{lemma}

\begin{proof} Consider the distribution
\[
\displaystyle u=\sum_{x \in \cL_M \cap I} \del_x,
\]
where $\del_x$ denotes the Dirac delta function at $x$, namely $\delta_x=\tau_{x}\delta$. This is the distribution that sends $g \in \cS(\bR)$ to $\sum_{x \in \cL_M \cap I} g(x)$. By the Parseval formula \eqref{Parseval} and Lemma \ref{periodicity}, we have 
\begin{equation} \label{ApplyParseval}
\sum_{x \in \cL_M \cap I \mmod 1} f(x) = \sum_{w \in \bZ} \tau_w u(f) = \sum_{w \in \bZ} \widehat {\tau_w u}(f^\vee)=\sum_{w\in\mathbb{Z}}\tau_w\delta(Ff^\vee),
\end{equation}
where $F$ is the function type of $\hat{u}$ and $f^\vee$ is the inverse Fourier transform of $f$; more explicitly, we recall that
\[
f^\vee(t)=\hat{f}(-t).
\] 
Via the change of variables $\beta=\alpha+y-\frac{b}{2^n}$ and an application of Fubini's Theorem, we compute that
\begin{align*}
\hat f(t) &= \int_\bR f(\alp) e(- t \alp) \d \alp \\[2ex]
            &= e(ty) \sum_{b  = 0}^{2^n-1} e(-t b/2^n)  \int_\bR \phi_r(\bet) e(-t \bet) \d \bet.
\end{align*} 
Observe that if $t \in \bZ$ then, employing the change of variables $\alpha' = r \beta$, we have 
\begin{equation} \label{IntegerFrequencies}
\hat f(t) = e(ty) 2^n 1_{2^n|t} r^{-1} \hat \phi(t/r) = 2^{-k} e(ty) \hat \phi(t/r) 1_{2^n|t}.
\end{equation}

The function type of $\hat{\delta_x}$ is $e(-x\cdot)$, since
\[
\hat{\del_x}(g) = \del_x(\hat g) = \hat g(x) 
= \int_\bR g(t) e(-tx) \d t
\]
for $g \in \cS(\bR)$. Thus, by the linearity of the Fourier transform, the function type $F$ of $\hat u$ is
\[F(t) 
       = \sum_{x \in \cL_M \cap I} e(-tx)\\[2ex]
       = \sum_{\substack{\eps_1, \ldots, \eps_M \in \{0,2\} \\  \sum_{j \le M} \eps_j/3^j \in I}} \prod_{j \le M} e(-t \eps_j/3^j).
\]
As $L \ge L_0(I)$, we see that if $\eps_1, \eps_2,\ldots \in \{0,2\}$ then
\[
\sum_{j =1}^\infty \frac{\eps_j}{3^j} \in I \quad \Longleftrightarrow \quad \sum_{j \le L} \frac{\eps_j}{3^j} \in I.
\]
In particular $\sum_{j \le M} \eps_j/3^j \in I$ if and only if $\sum_{j \le L} \eps_j / 3^j \in I$.
Therefore
\begin{align*}
F(t) &= \sum_{\substack{\eps_1, \ldots, \eps_M \in \{0,2\}: \\  \sum_{j \le M} \eps_j/3^j \in I}} e\left(-t \sum_{j=1}^{M}{\eps_j/3^j}\right) \\
  &= \sum_{\substack{\eps_1, \ldots, \eps_L \in \{0,2\}: \\  \sum_{j \le L} \eps_j/3^j \in I}} \quad \sum_{\eps_{L+1},\ldots,\eps_M \in \{0,2\}} {e\left(-t\sum_{j=1}^{L}{\frac{\varepsilon_j}{3^j}}\right) e\left(-t\sum_{j=L+1}^{M}{\frac{\varepsilon_j}{3^j}}\right)}.\end{align*}
Since the elements of $\cL_L \cap I$ are precisely the $\sum_{j \le L} \eps_j/3^j \in I$ with $\eps_j \in \{0,2\}$ for all $j$, note that
\[
\sum_{\substack{\eps_1, \ldots, \eps_L \in \{0,2\}: \\  \sum_{j \le L} \eps_j/3^j \in I}}  e\left(-t\sum_{j\le L}{\frac{\varepsilon_j}{3^j}}\right) = \sum_{x \in \cL_L \cap I} e(-tx).
\]
Next, we observe that
\begin{align*}
\sum_{\eps_{L+1},\ldots,\eps_M \in \{0,2\}} e\left(-t\sum_{j=L+1}^{M} \frac{\varepsilon_j}{3^j} \right) &= \sum_{\eps_{L+1},\ldots,\eps_M \in \{0,2\}} 
\prod_{j=L+1}^M e(-t \eps_j/3^j) \\ 
                                                   &= \prod_{j=L+1}^M \sum_{\eps_j \in \{0,2\}} 
e(-t \eps_j/3^j) 
                                                   = \prod_{j = L+1}^M (1 + e(-2t/3^j)) \\
                                                   &=2^{M-L} \prod_{j = L+1}^M \frac{1 + e(-2t/3^j)}2.
\end{align*}
Hence, we obtain
\[
F(t) = 2^{M-L} \sum_{x \in \cL_L \cap I} e(-tx) \prod_{j = L+1}^M \frac{1 + e(-2t/3^j)}2.
\]

By \eqref{ApplyParseval} and \eqref{IntegerFrequencies}, we now have
\begin{align*}
\sum_{x \in \cL_M \cap I \mmod 1} f(x) &= \sum_{w \in \Z}{\tau_{w}\delta(Ff^{\vee})} =\sum_{w \in \Z}{F(w)\hat{f}(-w)}  \\
&=2^{M-L-k} 
\sum_{w \in \bZ} e(-wy)\sum_{x \in \cL_L \cap I} e(- wx) \prod_{j=L+1}^M \frac{1+ e(-2w/3^j)}2  \hat \phi(-w/r) 1_{2^n|w}.
\end{align*}
Substituting $w = -2^n m$ and recalling that $r = 2^{n+k}$, we obtain 
\[
\sum_{x \in \cL_M \cap I \mmod 1} f(x) = 2^{M-L-k} \sum_{m\in\mathbb{Z}} \hat{\phi}(m/2^k) e(2^n m y) \sum_{x \in \cL_L \cap I} e(2^n m x) \prod_{j = L+1}^M \frac{1 + e(2^{n+1}m/3^j)}2.
\]
Let us now estimate the part of the above sum involving $m$ such that $|m|>2^k T.$ First, we see that 
\[
\left| \sum_{x \in \cL_L \cap I} e(2^n m x) \prod_{j = L+1}^M \frac{1 + e(2^{n+1}m/3^j)}2\right|\leq |\cL_L\cap I|.
\]
Next, we use the fact that $\phi$ is Schwartz, and thus has the rapid Fourier decay property \eqref{RapidDecay}, to deduce that for each number $N>1$,
\begin{align*}
\sum_{|m|>2^k T} |\hat{\phi}(m/2^k)| &\ll_N \sum_{m > 2^kT} (m/2^k)^{-N}  = 2^{kN} \sum_{m > 2^kT} m^{-N} \\
                                     &\le 2^{kN} \int_{2^kT}^\infty m^{-N} \d m = 2^{kN} \frac{(2^kT)^{1-N}}{N-1} \leq \frac{2^k}{T^{N-1}},
\end{align*}
where the subscript $N$ indicates that the implied constant depends on $N$. From here the proof concludes.
\end{proof}

To control the product term in Lemma \ref{FOURIERCOUNT}, we use digit changes in base 3. Set 
\[
\rho = | 1 + e(1/9) | < 0.94
\]
throughout.

\begin{lemma} \label{DIGITCHANGE} Let $n, L, M \in \bN$ with $L \le M$, and let $m \in \bZ \setminus \{ 0 \}$. Then
\[
\Biggl| \prod_{j = L+1}^M \frac{1+e(2^{n+1}m/3^j)}{2} \Biggr| \leq \rho^{D_3(2^{n+1}|m|; L, M)},
\]
where $D_3(2^{n+1}|m|; L, M)$ is the number of digit changes in the ternary expansion of $2^{n+1}|m|$ between the $L$'th and $M$'th digits, counting from the left. Note that when $L=M$, both sides in the above inequality are equal to $1$.
\end{lemma}

\begin{proof} We may assume, without loss of generality, that $m \in \bN$. Each term in the product has norm at most $1$. For $j \ge 2$, observe that $1+e(2^{n+1}m/3^j)$ has norm being bounded away from $2$ if the fractional part of $2^{n+1} m/3^j$ is not too close to $0$ or $1$. Suppose that the $(j-1)$'st and $j$'th ternary digits of $2^{n+1} m$ are different. Then the ternary expansion of $2^{n+1}m/3^j$ is
\[
2^{n+1} m/3^j= [\text{integer part}]. ab\dots,
\]
where $ab\in\{01,02,10,12,20,21\}$. This means that the fractional part of $2^{n+1}m/3^j$ is bounded away from $0$  and $1$. Indeed, it is easy to check that $\|\{2^{n+1}m/3^j\}\| \geq 1/9$ in this scenario.  Consequently, we have
\[
\left|\frac{1+e(2^{n+1}m/3^j)}{2}\right| \le \rho.
\]
The lemma now follows from the definition of $D_3(2^{n+1}|m|; L, M)$.
\end{proof}

Finally, we obtain the following estimate. 

\begin{lemma} \label{FinalCount} Let $\tet \in \bR$, and let $I$ be either $[0,1]$ or a subinterval of $[0,1]$ whose centre lies in $K$. Let $L$ be the maximum of the values $L_0(I)$ in Lemmas \ref{MeasureOfInterval} and \ref{FOURIERCOUNT}. Let $n \ge n_0(I)$ be a large positive integer, and let $k$ be an integer in the range
\[
1\le k \leq  \frac{\eps  \log n}{\log \log n},
\]
where $\eps >0$ is a sufficiently small, effectively-computable constant. In particular, $n$ should be sufficiently large that Lemma \ref{digit changes lemma} can be applied, and so that
\[
L+6 \leq \frac {c\log n}{4 \log \log n} \qquad \text{and}
\qquad 2^n \ge 3^{L+5},
\]
where $c$ is the implicit constant in Lemma  \ref{digit changes lemma}.

\begin{enumerate} [(a)]
\item If $M$ is the positive integer satisfying
\begin{equation} \label{M1}
3^{-5-M} < 2^{-n-k} \le 3^{-4-M},
\end{equation}
then
\[
|\cC_M \cap 0.2A_n(M)  \cap 0.2 I | \gg 2^{M-k} \mu(I).
\]
\item If $M$ is the positive integer satisfying
\begin{equation}\label{M2}
3^{5-M} < 2^{-n-k} \le 3^{6-M},
\end{equation}
then
\[
|\cC_M  \cap (45A_n(M)+\tet) \cap 5I | \ll 2^{M-k} \mu(I).
\]
\end{enumerate} 

The values of $\eps$ and the implicit constants do not depend on $\theta, n, k, I$.
\end{lemma}

\begin{proof} 
Recall that $\mu$ has the doubling property so, if $I \ne [0,1]$ then by rescaling $I$, we may assume that its endpoints do not lie in the Cantor set. To prove part (a), it suffices to prove that
\[
|\cC_M \cap 0.2A_n(M) \cap I | \gg 2^{M-k} \mu(I).
\]
This can be seen via a similar argument to that used at the beginning of the proof of Lemma \ref{MeasureOfInterval}.

By the construction of the bump function $\phi$, assuming \eqref{M1} we have
\begin{equation} \label{minorant}
|\cC_M \cap 0.2A_n(M) \cap I | \geq |\cL_M \cap 0.2A_n(M) \cap I | \ge \sum_{x \in \cL_M \cap I \mmod 1} f_n(x;0),
\end{equation}	
with the notation of Lemma \ref{FOURIERCOUNT}. Indeed, as $\phi_r$ is supported on $[-2^{1-n-k},2^{1-n-k}]$ and has image in $[0,1]$, the right-hand side is bounded above by the number of pairs $(b,x) \in \{0,1,\ldots,2^n\} \times (\cL_M \cap I)$ such that
\[
\Bigl| x - \: \frac b{2^n} \Bigr| \le \frac2{3^{M+4}},
\]
and this count is in turn is bounded above by $|\cL_M \cap 0.2A_n(M) \cap I |\leq |\cC_M \cap 0.2A_n(M) \cap I |$. 

Using Lemma \ref{FOURIERCOUNT} with $y=0$ and a suitably large fixed $N$, together with \eqref{minorant}, we see that for some constant $c_N>0$ we have 
\begin{align*}
|\cC_M \cap 0.2A_n(M) \cap I | &\geq 2^{M-L-k} \left(\sum_{|m|\leq 2^k T} \hat{\phi}(2^{-k}m)\sum_{x \in \cL_L \cap I} e(2^n m x) \prod_{j = L+1}^M \frac{1 + e(2^{n+1}m/3^j)}2 \right)\\
&-2^{M-L}c_N\frac{|\cL_L\cap I|}{T^N},
\end{align*} 
where $T\geq 1$ and $M \ge L$. The zero frequency, i.e. $m=0$, contributes
\[
2^{M-L-k}\hat{\phi}(0) |\cL_L \cap I|=c_1 2^{M-k}\mu(I),
\]
where $c_1>0$ and the above equality can be seen as the definition of $c_1$. Moreover, by Lemma \ref{MeasureOfInterval}, 
\[
2^{M-L-k}\hat{\phi}(0) |\cL_L \cap I| \asymp 2^{M-k} \mu(I).
\]
Thus we see that $c_1$ is bounded (from above and away from zero) by absolute constants, i.e. $E^{-1}<c_1<E$ for some absolute constant $E>1.$

 For the remaining terms, we first deal with $2^{M-L}c_N|\cL_L\cap I|/T^N.$ We choose $T$ to be such that
\[
2^{M-L}c_N\frac{|\cL_L\cap I|}{T^N}\leq 0.5c_12^{M-k}\mu(I).
\] 
For the above to hold, we need 
\[
T^N\geq \frac{2c_N}{\hat{\phi}(0)}2^k. 
\]
We set the value of $\log T$ to be $\max\left\{0, \frac{k\log 2}{N}+\frac{\log (2c_N/\hat{\phi}(0))}{N}\right\}.$ That is, we choose
\[
T=\max\{1, (2c_N/\hat{\phi}(0))^{1/N}2^{k/N}\}.
\]

Now for the sum with $0<|m|\leq 2^k T$, the triangle inequality gives 
\begin{align*}
&\left|\sum_{0 < |m|\leq 2^k T} \hat{\phi}(2^{-k}m)
\sum_{x \in \cL_L \cap I} e(2^n m x) \prod_{j = L+1}^M \frac{1 + e(2^{n+1}m/3^j)}2\right|
 \\
&\leq \hat{\phi}(0) |\cL_L \cap I| \sum_{0 < |m|\leq 2^k T}\left|\prod_{j= L+1}^M \frac{1+e(2^{n+1}m/3^j)}{2}\right|.
\end{align*}
In light of Lemma \ref{MeasureOfInterval}, this is bounded above by a constant times
\[
2^L \mu(I) \sum_{0 < |m|\leq 2^k T}\left|\prod_{j= L+1}^M \frac{1+e(2^{n+1}m/3^j)}{2}\right|.
\]

Next, we will apply Lemma \ref{DIGITCHANGE}, and so we need to estimate $D_3(2^n|m|;L,M)$. First recall, by Lemma \ref{digit changes lemma}, that for some constant $c>0$ we have
\begin{align} \label{digit changes consequence}
D_2(2^n |m|)+D_3(2^n |m|)\geq c\frac{\log\log{(2^n |m|)}}{\log\log\log{(2^n |m|)}},
\end{align}  
and note that
\[
D_3(2^n |m|; L, M)\geq D_3(2^n |m|)- (\Delta_3(2^n |m|)-M + L),
\] 
where $\Delta_3(2^n |m|)$ is the number of ternary digits of $2^n| m|$. From \eqref{M1}, we see that \mbox{$M+5= \Del_3(2^{n+k})$.} As  $2^{n}|m|\leq 2^{n+k} T$ (since $|m|\leq 2^k T$), we must therefore have \[\Delta_3(2^n|m|) \leq \Delta_3(2^{n+k}T) \leq \Delta_3(2^{n+k}) + \Delta_3(T),\] and consequently $\Delta_3(2^n |m|) - M \le \Delta_3(T)+5$. 
Hence
\[
D_3(2^n |m|; L, M)\geq D_3(2^n |m|) - \Delta_3(T) - L - 5.
\]
Furthermore
\[
D_2(2^n |m|) \le 1 + D_2(|m|) \leq 1+ \frac{\log |m|}{\log 2}.
\]
Using \eqref{digit changes consequence}, we now see that
\begin{align}
\notag D_3(2^n |m|; L, M)&\geq D_{3}(2^n|m|) - \Delta_3(T) - L-5 \\
 \notag                 &\geq \frac{c\log{\log{(2^n|m|)}}}{\log{\log{\log{(2^n|m|)}}}}-D_2(2^n|m|)-\Delta_3(T)-L-5 \\                    
  \label{D3bound}                &\geq \frac{c\log\log{(2^n |m|)}}{\log\log\log{(2^n |m|)}}- \frac{\log |m|}{\log 2} -6-\Delta_3(T) - L.
\end{align}
As $\log \log(\cdot) / \log \log \log(\cdot)$ is increasing, we may replace its argument by $2^n$ for a lower bound, to see that the first term on the right-hand side of \eqref{D3bound} is at least $\frac{c\log n}{2\log \log n}$. Meanwhile, the upper bound on $|m|$ assures us that $\log |m| \le k\log 2+ \log T$. Therefore
\[
D_3(2^n |m|; L, M) \ge \frac{c\log n}{2\log \log n} - k- \frac{\log T}{\log 2}-\Delta_3(T) - L - 6.
\]
Since $n \ge n_0(I)$, we may assume that $n$ is arbitrarily large compared to $L$, so that $L+6 \le \frac {c \log n}{4 \log \log n}$ and, hence
\[
D_3(2^n |m|; L, M) \ge \frac {c \log n}{4 \log \log n} - k-\frac{\log T}{\log 2}-\Delta_3(T).
\]
Moreover, \[\Delta_3(T)=\lfloor\log T/\log 3\rfloor +1 \leq \log T +1\leq \log T+k.\]
Recalling that \mbox{$\log T=\max\left\{0, \frac{k\log 2}{N}+\frac{\log (2c_N/\hat{\phi}(0))}{N}\right\}$} as well as $k\geq 1$ and noting that $\hat{\phi}(0)\geq2$  we now see that
\begin{align*}
D_3(2^n|m|;L,M) &\geq \frac{c\log{n}}{4\log{\log{n}}} - 2k - 3\log{T}\\ 
                &\geq \frac{c\log{n}}{4\log{\log{n}}} - k\left(2+\frac{3\log{2}}{N} + \frac{3\log\left(\frac{2c_N}{\hat{\phi}(0)}\right)}{N}\right) \\
                &\geq \frac{c\log n}{4\log\log n} -\left(2+3\frac{\log 2}{N}+3|\log (2c_N)|\right)k.
\end{align*}

For convenience, we write
\[
\kappa_{n,N,k}=\frac{c\log n}{4\log\log n} -\left(2+3\frac{\log 2}{N}+3|\log (2c_N)|\right)k.
\]
Now Lemma \ref{DIGITCHANGE} gives
\[
\sum_{0 < |m|\leq 2^k T}\left|\prod_{j\leq M} \frac{1+e(2^{n+1}m/3^j)}{2} \right| \ll 2^k T\rho^{\kappa_{n,N,k}}.
\] 
Recalling that \[
T=\max\left\{1, \left(\frac{2c_N}{\hat{\phi}(0)}\right)^{1/N}2^{k/N}\right\}\leq 1+\left(\frac{2c_N}{\hat{\phi}(0)}\right)^{1/N}2^{k/N},
\]
we see that
\[
2^{-k}\sum_{0 < |m|\leq 2^k T}\left|\prod_{j\leq M} \frac{1+e(2^{n+1}m/3^j)}{2} \right| \ll \rho^{\kappa_n,N,k}+\left(\frac{2c_N}{\hat{\phi}(0)}\right)^{1/N}2^{k/N}\rho^{\kappa_{n,N,k}}.
\]
Observe that 
\[
2^{k/N}\rho^{\kappa_{n,N,k}}=\rho^{\kappa_{n,N,k}+k\log 2/(N\log \rho)    }.
\]
Again, for convenience, we write
\[
\kappa'_{n,N,k}=\kappa_{n,N,k}+k\log 2/(N\log \rho)  .
\]
Thus, the non-zero frequencies contribute at most
\[
C \cdot 2^{M-L} 2^L \mu(I) {\rho}^{\kappa'_{n,N,k}} = C\cdot2^M \mu(I){\rho}^{\kappa'_{n,N,k}},
\]
for some constant $C>0$. Recalling that the contribution from the zero frequency is $c_1 2^{M-k} \mu(I)$, for some $c_1$ bounded below by a positive absolute constant, we glean that
\begin{align}
\notag|\cC_M \cap 0.2A_n(M) \cap I |  &\ge c_1 2^{M-k} \mu(I) - C \cdot 2^M \mu(I){\rho}^{\kappa'_{n,N,k}}- 0.5 c_1 2^{M-k} \mu(I) \\
\label{Es}&\ge 2^{M-k} \mu(I) (0.5c_1 - C{\rho}^{\kappa'_{n,N,k}+ k\log 2/\log \rho}).
\end{align}

We have
\[\kappa'_{n,N,k}+k\frac{\log{2}}{\log{\rho}} = \frac{c\log{n}}{4\log{\log{n}}} - \left(2+\frac{3\log{2}}{N}+3|\log(2c_N)|-(1+N^{-1})\frac{\log{2}}{\log{\rho}}\right)k.\]
Next, we write
\[
c'_N=2+\frac{3\log 2}{N}+3|\log (2c_N)|-(1+N^{-1})\frac{\log{2}}{\log{\rho}}.
\]
Then we choose the value of $\varepsilon$ to be 
\[
\varepsilon=\frac{c}{8c'_N}.
\]
As $k \le \eps \log n / \log \log n$, this ensures that
\[
\kappa'_{n,N,k}+k\frac{\log 2}{\log \rho} \geq \frac{c\log n}{8\log\log n}.
\]
Recalling that $\rho\in (0,1)$ is an absolute constant and that $c_1$ is bounded away from 0, the result concludes by using \eqref{Es}.\\

The second inequality can be proved similarly: the zero frequency again dominates. As Lemma~\ref{FOURIERCOUNT} applies to $\cL_M$ and not $\cC_M$, we note that
\[
|\cC_M  \cap (45A_n(M)+\tet) \cap 5I | 
= |\cL_M  \cap (45A_n(M)+\tet) \cap 5I | 
+ |\cR_M  \cap (45A_n(M)+\tet) \cap 5I |,
\]
and by symmetry that
\[
|\cR_M  \cap (45A_n(M)+\tet) \cap 5I | 
= |\cL_M \cap (1-(45A_n(M)+\tet)) \cap (1-5I)|.
\]
Moreover, observe that
\[
1-(45A_n(M)+\tet) = (1 - 45A_n(M)) -  \tet = 45A_n(M) -\tet.
\]
Thus, we may also estimate
\[
|\cR_M  \cap (45A_n(M)+\tet) \cap 5I |
= |\cL_M \cap (45A_n(M) -\tet) \cap (1-5I)|
\]
in the same way using Lemma \ref{FOURIERCOUNT}, with $1-I$ in place of $I$, since $1-5I$ is the dilation of the interval $1-I$ by a factor of 5 about its midpoint.
\end{proof}

The following technical lemma enables us to slightly relax the hypotheses \eqref{M1} and \eqref{M2}.

\begin{lemma}\label{shifting parameters}
Let $\tet \in \bR$, and let $I$ be a real interval. Let $k, M, N, J,t \in \N$. Then, supposing that $0 \leq J < M$ and $2^{-n}> 2t/3^{M-J},$ we have
\[|\cC_{M-J}\cap (tA_n(M-J)+\theta) \cap I|+ 1\gg |\cC_{M}\cap (t A_n(M)+\theta) \cap I| ,\]
where the implied constant depends only on $J$. If $I\supseteq [0,1]$, then the $+1$ term on the left-hand side can be removed.
\end{lemma}

\begin{proof} 
First, assume that $I \supseteq [0,1].$ Elements of $\cC_{M-J}$ are endpoints of the intervals of length $3^{-(M-J)}$ forming $K_{M-J}.$ We see from the inequality $2^{-n} > 2t/3^{M-J}$ that $tA_n(M-J)+\theta$ is a disjoint union of intervals of length $2t/3^{M-J}$ centred at shifted dyadic rationals in $[0,1]$ with denominator $2^n.$ As $2t/3^{M-J} > 1/3^{M-J},$ observe that for each interval forming $K_{M-J}$ and intersecting \mbox{$tA_n(M-J)+\theta$,} at least one (and trivially at most two) of the endpoints must be included in $tA_n(M-J)+\theta.$ Writing $\cN(M-J)$ for the number of such intervals, and $\cN(M)$ for the number of intervals forming $K_M$ and intersecting $tA_n(M-J)+\theta$, and employing the same argument for $\cC_M$ as for $\cC_{M-J}$, we therefore have
\[
1 \le \frac{|\cC_{M-j} \cap (tA_n(M-J) + \theta)|}{\cN(M-j)} \le 2 \qquad (j=0,J).
\]
Let $I_M$ be such an interval counted by $\cN(M)$. Then $I_M\subset I_{M-J}$ for some interval $I_{M-J}$ forming $K_{M-J}$ for which $I_{M-J}\cap (tA_n(M-J)+\theta)\neq\emptyset.$ Clearly, each interval forming $K_{M-J}$ contains $2^{J}$ many intervals forming $K_M.$ From here we see that
\begin{align*}
|\cC_{M}\cap (tA_n(M)+\theta)|
&\leq |\cC_{M}\cap (tA_n(M-J)+\theta)| \\
&\le 2 \cN(M) \\
& \le 2^{J+1} \cN (M-J) \\
&\leq 2^{J+1}|\cC_{M-J}\cap (tA_n(M-J)+\theta)|,
\end{align*}
where the first inequality follows because $tA_n(M)+\theta \subseteq tA_n(M-J)+\theta$.

Now, let $I$ be a general subinterval of $\bR$. As above, we find that $|\cC_M\cap (tA_n(M-J)+\theta)\cap I|$ is at most double the number of intervals forming $K_M$ and intersecting $(tA_n(M-J)+\theta)\cap I.$ Let $I_M, I_{M-J}$ be as in above. The additional consideration is that endpoints of $I_{M-J}$ may not be in $I.$ However, this can happen for at most two intervals forming $K_{M-J}.$ Thus, writing $\cN(M-J;I)$ for the number of intervals forming $K_{M-J}$ and intersecting $(tA_n(M-J)+\theta)\cap I,$ we have
\[
\cN(M-J;I) \leq |\cC_{M-J}\cap (tA_n(M-J)+\theta)\cap I|+2.
\]
Finally, we obtain
\begin{align*}
|\cC_{M}\cap (tA_n(M)+\theta)\cap I| &\le |\cC_{M}\cap (tA_n(M-J)+\theta)\cap I| \\
                                     &\leq 2^{J+1} \cN(M-J;I)\\
                                     &\leq 2^{J+1} (|\cC_{M-J}\cap (tA_n(M-J)+\theta)\cap I|+2). \qedhere
\end{align*}
\end{proof}

\section{Convergence theory} \label{CT}

Here we establish Theorem \ref{main convergence theorem}. As $\sum_{n=1}^\infty \psi(2^n) < \infty$, we may assume that $\psi(2^n) < 3^{-99}$ for $n$ sufficiently large. Let $n$ be such a large positive integer, such that additionally $\log{\log{\log{n}}} > 3/\eps$, and Lemmas \ref{connection}, \ref{DroppingDown}, and \ref{FinalCount} may be applied, where $\eps$ is from Lemma \ref{FinalCount}. Next, let
\[
k_{n} = \min\left(3\Bigl \lfloor \frac{\log n}{\log \log n \cdot \log \log \log n} \Bigr \rfloor, \Bigl \lfloor \frac{-\log \psi(2^n)}{\log 2} \Bigr \rfloor+1\right),
\]
and choose $M, N \in \bN$ according to
\begin{equation} \label{MN}
3^{-5-M} < 2^{-n-k_{n}} \le 3^{-4-M} \qquad \text{and} \qquad 3^{1-N} < \frac{\psi(2^n)}{5 \cdot 2^n} \le 3^{2-N}.
\end{equation}
These choices ensure that we have all of the inequalities in \eqref{ValidDrop}, and also the inequalities
\[
1\leq k_n \le \frac{\eps \log n}{\log \log n},
\]
which will be needed when we apply Lemma \ref{FinalCount}. Now, fix $I=[0,1]$ and observe that 
\[\mu(A_n)=\mu(A_n \cap I).\]
By Lemma \ref{connection}, we have
\[\mu(A_n \cap I) \leq 2^{-(N-1)}|\cC_N \cap 5A_n \cap 5I|.\]
Applying Lemma \ref{DroppingDown}, we deduce that
\[|\cC_N \cap 5A_n \cap 5I| \ll |\cC_M \cap 5A_n(M) \cap 5I| \leq |\cC_M \cap 45A_n(M) \cap 5I|.\]
Next, we apply Lemma \ref{FinalCount}(b) with $\tet = 0$. Our current parameters $M,n,k=k_n$ do not satisfy the conditions of \eqref{M2} required to apply Lemma \ref{FinalCount}(b). However, the parameters $M-10,n,k_n$ do satisfy \eqref{M2}, and so
\[
|\cC_{M-10}\cap 45A_n(M-10)\cap 5I|\ll 2^{M-10-k_n}\mu(I)
\ll 2^{M-k_n} \mu(I).
\] 
Now we use Lemma \ref{shifting parameters} to deduce that
\[
|\cC_{M}\cap 45A_n(M)\cap 5I|\ll |\cC_{M-10}\cap 45A_n(M-10)\cap 5I|,
\]
where the implicit constant is absolute; here we note from the calculation
\[
3^{-5-M} < 2^{-n-k_n} < \frac{\psi(2^n)}{2^n} < \frac{3^{-99}}{2^n}
\]
that the required condition $2^{-n} > \frac{90}{3^{M-10}}$ is met.

Therefore
\[\mu(A_n) \ll 2^{-N}|\cC_M \cap 45A_n(M) \cap 5I| \ll 2^{M-k_n-N}\mu(I) = 2^{M-k_n-N}.\]
Moreover, since $3^M \ll 2^{n+k_n}$ and $3^{-N} \ll \frac{\psi(2^n)}{2^n}$ by hypothesis, we have
\[2^{M - N -k_n} = (3^M \times 3^{-N})^\gamma \times 2^{-k_n} \ll \left(2^{n+k_n}\times\frac{\psi(2^n)}{2^n}\right)^{\gamma} \times 2^{-k_n} \ll (2^{k_n} \psi(2^n))^\gamma \times 2^{-k_n},\]
and hence
\[\mu(A_n) \ll (2^{k_n} \psi(2^n))^\gamma \times 2^{-k_n}.\]

Observe that if \[k_n=\Bigl \lfloor \frac{-\log \psi(2^n)}{\log 2} \Bigr \rfloor+1\]
then \[(2^{k_n} \psi(2^n))^\gamma \times 2^{-k_n}\ll \psi(2^n).\]
Otherwise, we have
\[
(2^{k_n} \psi(2^n))^\gamma \times 2^{-k_n}= 2^{-3(1-\gamma)\bigl \lfloor \frac{\log n}{\log \log n \cdot \log \log \log n} \bigr \rfloor}\psi(2^n)^{\gamma}\ll 2^{-\log n/(\log\log n\cdot\log\log\log n)}\psi(2^n)^\gamma.
\]
From the above arguments we see that
\[
\sum_{n=1}^{\infty} \mu(A_n) \ll \sum_{n=1}^\infty 2^{-k_n (1-\gam)} \psi(2^n)^\gamma \ll \sum_{n=1}^{\infty} (2^{-\log n / (\log \log n \cdot \log \log \log n)} \psi(2^n)^\gamma+\psi(2^n)),
\]
which converges by our assumption \eqref{main convergence theorem sum}. Thus, by the convergence Borel--Cantelli Lemma (Lemma~\ref{first Borel-Cantelli}), the proof of Theorem \ref{main convergence theorem} is complete.

\section{Divergence theory}\label{divergence theory proof}

Here we establish Theorem \ref{main divergence theorem}. Let $I = B(z,r)$, for some $z \in K$ and some $r \in (0,1)$. Define the `localised' probability measure $\mu_I$ by
\[
\mu_I(A) = \frac{\mu(A \cap I)}{\mu(I)},
\]
for Borel sets $A$. By Lemma \ref{positive measure lemma}, recalling that $\mu$ has the necessary properties as stated at the beginning of \S \ref{FourierAnalysis}, it suffices to prove that
\begin{equation} \label{LocalLower}
\mu_I(W_2(\psi)) \gg 1,
\end{equation}
with an implicit constant independent of $I$.\\

We begin by showing that if $n \ge n_0(I)$, where $n_0(I)$ is a sufficiently large integer such that we can apply Lemmas \ref{connection} and \ref{FinalCount}, then $\mu_I(A_n) \gg \psi(A_n)$, where the implicit constant does not depend on $I$. Define $N, k \in \bN$ according to
\[
3^{-N} \le \frac{\psi(2^n)}{5 \cdot 2^n} < 3^{-(N-1)} \qquad \text{and} \qquad 3^{-5-N} < 2^{-n-k} \le 2 \cdot 3^{-5-N}.
\]
Now, since $3^{-N} \leq \frac{\psi(2^n)}{2^n}$, we have \[|\cC_N \cap 0.2A_n \cap 0.2I| \geq |\cC_N \cap 0.2A_n(N) \cap 0.2I|.\]
Combining this with Lemma \ref{connection} yields
\[\mu_I(A_n) \ge \frac{2^{-(N+1)}|\cC_N \cap 0.2A_n(N) \cap 0.2I|}{\mu(I)}.\]
Thus, it follows from Lemma \ref{FinalCount}(a) that
\[\mu_I(A_n) \gg \frac{2^{-(N+1)} \times 2^{N-k} \mu(I)}{\mu(I)} \gg 2^{-k}.\] 
As 
\[
2^{-k} \gg 2^n 3^{-N} \gg \psi(2^n) = 2^{- \log \log n / \log \log \log n},
\] 
we have
\begin{align} \label{A_n divergence estimate}
\mu_I(A_n) &\gg \psi(2^n),
\end{align}
as claimed.

\bigskip

Next, we wish to bound $\mu_I(A_n \cap A_m)$. We will do this analytically when
\[
n_0(I) \le n \le m \le n^+,
\]
for $n^+$ to be specified in due course. By the triangle inequality, if $B\left(\frac{a}{2^n}, \frac{\psi(2^n)}{2^n}\right)$ intersects $B\left(\frac{b}{2^m}, \frac{\psi(2^m)}{2^m}\right)$ then
\[
\Bigl| \frac a{2^n} - \frac b {2^m} \Bigr| < \frac {\psi(2^n)}{2^n} + \frac{\psi(2^m)}{2^m},
\]
whereupon
\begin{align} \label{divergence auxiliary inequality}
|b - 2^{m-n}a| &< 2^{m-n} \psi(2^n) + \psi(2^m).
\end{align}
We break the solutions $(a,b)$ to \eqref{divergence auxiliary inequality} into at most 
\[
2(1 + 2^{m-n} \psi(2^n) +  \psi(2^m)) \ll 1 + 2^{m-n} \psi(2^n)
\]
groups, according to the value of the integer $h = b - 2^{m-n}a$, the idea being to handle each group analytically.

For a given value of $h = b - 2^{m-n}a$, the corresponding intersections $B\left(\frac{a}{2^n}, \frac{\psi(2^n)}{2^n}\right) \cap B\left(\frac{b}{2^m}, \frac{\psi(2^m)}{2^m}\right)$ are contained in 
\begin{align*}
X_h:=\bigcup_{a=0}^{2^n} B\Bigl( \frac{a}{2^n} + \frac{h}{2^m}, \frac{\psi(2^m)}{2^m} \Bigr).
\end{align*}
To see this, observe that if
\[ x \in \bigcup_{a = 0}^{2^n} \bigcup_{b=0}^{2^m} B\left(\frac{a}{2^n}, \frac{\psi(2^n)}{2^n}\right) \cap B\left(\frac{b}{2^m}, \frac{\psi(2^m)}{2^m}\right),\]
then, for some $a \in \{0,1,2 , \ldots, 2^n\}$ and $b \in \{0, 1, 2, \ldots, 2^m\}$, we have 
\[x \in B\left(\frac{a}{2^n}, \frac{\psi(2^n)}{2^n}\right) \cap B\left(\frac{b}{2^m}, \frac{\psi(2^m)}{2^m}\right).\]
Next, write $h = b - 2^{m-n} a$. Then,
\[
x  \in B\left(\frac{b}{2^m}, \frac{\psi(2^m)}{2^m}\right) = B\Bigl( \frac{a}{2^n} + \frac{h}{2^m}, \frac{\psi(2^m)}{2^m} \Bigr) \subseteq X_h.
\]
Recalling from \eqref{divergence auxiliary inequality} that $|h| < 2^{m-n} \psi(2^n) + \psi(2^m)$, it follows that
\begin{align}
\notag \mu_I(A_n \cap A_m) &= \mu_I\left(\bigcup_{a = 0}^{2^n} \bigcup_{b=0}^{2^m} B\left(\frac{a}{2^n}, \frac{\psi(2^n)}{2^n}\right) \cap B\left(\frac{b}{2^m}, \frac{\psi(2^m)}{2^m}\right)\right) \\
\label{groups} &\le \sum_{|h| < 2^{m-n} \psi(2^n) + \psi(2^m)} \mu_I(X_h).
\end{align}

We will use our Fourier-analytic machinery from Section \ref{FourierAnalysis} to bound each $\mu_{I}(X_h)$. To begin with, the triangle inequality gives
\[
\mu_I(X_h) \le \sum_{a=0}^{2^n}
\mu \left( B \left( \frac{a}{2^n} + \frac{h}{2^m}, \frac{ \psi(2^m)}{2^m} \right) \right).
\]
Next, we apply \eqref{VersatileMeasure} to each of the intervals $B(a/2^n + h/2^m, \psi(2^m)/2^m) \cap I$ with $L = R$ therein, where $R$ is the positive integer satisfying
\[
3^{1-R} < \frac{\psi(2^m)}{2^m} \le 3^{2-R}.
\]
This gives
\[
\mu(X_h \cap I) \ll 2^{-R} |\cC_R \cap 5 X_h \cap 5I|,
\]
where
\[
5X_h = \bigcup_{a=0}^{2^n} B\Bigl( \frac{a}{2^n} + \frac{h}{2^m}, \frac{5\psi(2^m)}{2^m} \Bigr),
\]
since at most 11 of the balls can contain a given point. We thus have 
\begin{equation} \label{CRbound}
\mu(X_h \cap I) \ll 2^{-R} \left|\cC_R \cap \left(45A_n(R) + \frac{h}{2^m}\right) \cap 5I\right|.
\end{equation}

Let $k(m,n)$ be the positive integer for which
\[
2^{-k(m,n)-9} < \frac{\psi(2^m)}{2^{m-n}} = 2^{n-m-(\log \log m / \log \log \log m)} \leq 2^{-k(m,n)-8}.
\]
We will apply Lemma \ref{FinalCount}(b) with $M \in \bN$ given by
\[
3^{5-M} < 2^{-n-k(m,n)} \le 3^{6-M}.
\]
In order to meet the condition $k(m,n) \le \frac{\eps \log n }{\log \log n}$ required for this, it suffices to have
\[
m - n + \frac{ \log \log m}{\log \log \log m} = o \Bigl(\frac{ \log n}{\log \log n} \Bigr),
\]
and this is assured if we choose
\begin{align}\label{n+}
n^+ = n + \Biggl \lfloor \frac{\log n}{\log \log n \cdot \log \log \log n} \Biggr \rfloor.
\end{align}
To see this, we note that $m \mapsto\frac{\log{\log{m}}}{\log{\log{\log{m}}}}$ is increasing so, with this choice of $n^+$, we have
\begin{align*}
m-n+\frac{\log{\log{m}}}{\log{\log{\log{m}}}} &\leq n + \frac{\log{n}}{\log{\log{n}}\cdot\log{\log{\log{n}}}} - n +\frac{\log{\log{\left(n + \frac{\log{n}}{\log{\log{n}}\cdot\log{\log{\log{n}}}}\right)}}}{\log{\log{\log{\left(n + \frac{\log{n}}{\log{\log{n}}\cdot\log{\log{\log{n}}}}\right)}}}} \\
                                              &=o\left(\frac{\log{n}}{\log{\log{n}}}\right).
\end{align*}
Now Lemma \ref{FinalCount}(b) gives
\begin{equation} \label{CMbound}
\left|\cC_M \cap \left(45 A_n(M) + \frac{h}{2^m}\right) \cap 5I\right| \ll 2^{M-k(m,n)} \mu(I).
\end{equation}

We see from our choices of $M$, $R$, and $k(m,n)$ that
\[
3^{6-M} \ge 2^{-n-k(m,n)} \ge \frac{\psi(2^m)}{ 2^{m-8}} > 2^8 3^{1-R} > 3^{6-R}
\]
and
\[
3^{4-M} < \frac{3^{5-M}}{2} < 2^{-n-k(m,n)-1} < 2^8\frac{\psi(2^m)}{2^{m}} \le 2^8 3^{2-R} < 3^{8-R},
\]
so $R-4 < M < R$. It follows from the above inequalities that $2^{-n} > \frac{90}{3^M}$ and so we may apply Lemma~\ref{shifting parameters} with $R$ and $M$ playing the roles of $M$ and $M-J$ respectively, and this tells us that
\[
\left|\cC_R \cap \left(45 A_n(R) + \frac{h}{2^m}\right) \cap 5I\right| \ll
\left|\cC_M \cap \left(45 A_n(M) + \frac{h}{2^m}\right) \cap 5I\right| + 1.
\]
Combining this with \eqref{CRbound} and \eqref{CMbound} and the fact that $M<R$, we see that 
\begin{equation} \label{Xh}
\mu_I(X_h) \ll \frac{2^{M-k(m,n)} \mu(I) + 1}{2^M \mu(I)}
\ll 2^{-k(m,n)} + \frac{1}{2^M \mu(I)}.
\end{equation}

Note that
\[
3^{-M} \asymp \psi(2^m)/2^m
\]
and
\[
2^{-k(m,n)} \asymp 2^n/3^M \asymp 2^{n-m} \psi(2^m).
\]
Observe also that
\[
2^{-M} = (3^{-M})^\gam \asymp (2^{-k(m,n)-n})^\gam \ll 2^{-n\gam}.
\]
So, since
\[
k(m,n)=o\left(\frac{\log n}{\log\log n}\right),
\]
we must therefore have 
\[\frac{1}{2^M \mu(I)} \ll 2^{-n \gamma}/\mu(I) \ll 2^{-\log{n}} \ll 2^{-k(m,n)} \qquad \text{whenever } n \ge n_0(I).\] 
Substituting this data into \eqref{Xh} gives
\[
\mu_I(X_h) \ll 2^{-k(m,n)} \asymp 2^{n-m} \psi(2^m).
\]
Substituting this into \eqref{groups}, we finally obtain
\begin{equation} \label{intersection estimate}
\mu_I(A_n \cap A_m) \ll (1 + 2^{m-n} \psi(2^n) ) \times 2^{n-m} \psi(2^m) = 2^{n-m} \psi(2^m) + \psi(2^n) \psi(2^m).
\end{equation}

The Chung--Erd\H{o}s inequality (Lemma \ref{Chung-Erdos}) gives
\[
 \mu_I \Biggl( \bigcup_{i=n}^{n^+} A_i \Biggr) \ge \frac{ \displaystyle \Bigl(\sum_{i=n}^{n^+} \mu_I( A_i) \Bigr)^2} { \displaystyle \sum_{i,j=n}^{n^+} \mu_I(A_i \cap A_j) }.
\]
Estimating the denominator using \eqref{intersection estimate} and \eqref{A_n divergence estimate}, we have
\begin{align*}
\displaystyle \sum_{i,j=n}^{n^+} \mu_I(A_i \cap A_j) &\ll \displaystyle \sum_{i,j=n}^{n^+} \mu_I( A_i) \mu_I(A_j) + \sum_{n \le i \leq j \le n^+} \psi(2^i) 2^{i-j} \\
&\ll  \displaystyle \Bigl(\sum_{i=n}^{n^+} \mu_I( A_i) \Bigr)^2 + \displaystyle \sum_{i=n}^{n^+} \mu_I( A_i).
\end{align*}
For large $n$, by \eqref{A_n divergence estimate} we have 
\[
\sum_{t=n}^{n^+} \mu_I(A_t) \gg \sum_{t=n}^{n^+} 2^{-\log \log t/ \log \log \log t}.
\]
If $n$ is large and $n \le t \le n^+$, then $t \le 2n$ and by an application of the mean value theorem we have 
\[
\frac{\log \log t}{\log \log \log t} - \frac{\log \log n}{\log \log \log n} = (t-n) \frac{ \log \log \log c - 1}{c (\log c) (\log \log \log c)^2} \le 1,
\]
for some $c \in [n, n^+]$. Therefore
\[
\sum_{t=n}^{n^+} \mu_I(A_t) \gg 
\frac{ \log n}{\log \log n \cdot \log \log \log n} 
2^{-\log \log n / \log \log \log n},
\]
and this tends to $\infty$ as $n\to\infty$. Indeed, to see that $2^{\log \log n/\log \log \log n}$ is sub-logarithmic, observe that
\[
2^{\log \log n/\log \log \log n} = (\log n)^{\log 2 / \log \log \log n} = (\log n)^{o(1)}.
\]
We thus obtain
\[
 \mu_I\Biggl( \bigcup_{i=n}^{n^+} A_i \Biggr)  \gg \frac{1}{1+\frac{1}{\sum_{t=n}^{n^+} \mu_I(A_t)}} \gg 1.
\]

Finally, let
\[
B_n=\bigcup_{i=n}^{n^+} A_i.
\]
Using the continuity of the measure $\mu_I,$ we see that 
\begin{align*}
\mu_I\left(\limsup_{n\to\infty} A_n\right) &= \mu_I\left(\limsup_{n\to\infty} B_n\right) =\mu_I\left(\bigcap _{n\geq 1}\bigcup_{k\geq n} B_n\right) \\
                                &=\lim_{n\to\infty} \mu_I \left(\bigcup_{k\geq n} B_n\right) \geq \lim_{n\to\infty} \mu_I(B_n)\gg 1.
\end{align*}
This completes the proof that
\[
\mu_I(W_2(\psi)) = \mu_I\left( \limsup_{n \to \infty} A_n\right) \gg 1,\]
which is \eqref{LocalLower}.

\section{Further discussion} \label{conditional results section}

\subsection{Digit changes to different bases}

The ``times two, times three'' phenomenon is, roughly speaking, the mantra that digit expansions to two multiplicatively independent bases cannot both be structured. One way to quantify this is to bound from below the maximum --- or equivalently the sum --- of $D_2(y)$ and $D_3(y)$, since $D_b(y)$ being small means that the base $b$ expansion of $y$ is structured.

\begin{qn}\label{RGSC} Is it true that for sufficiently large $y \in \N$, we have
	\[
	D_2(y) + D_3(y) \gg \log y?
	\]
\end{qn}

\noindent Note that for all $y \in \bN$ we have $D_2(y) + D_3(y) \le 10 (1 + \log y)$. This bound can be obtained using ``naive'' estimates: observing that $y$ has at most $1 + \frac{\log{y}}{\log{2}}$ digits in its base 2 expansion and at most $1 + \frac{\log{y}}{\log{3}}$ digits in base 3. Consequently, the number of digit changes in base 2 and base 3 are also bounded above by these values, respectively. Finally, by estimating $\log{2}$ and $\log{3}$ quite crudely, e.g. $\frac{1}{2} < \log{2},\log{3} < 2$ will do, we easily obtain the claimed upper bound. For Question~\ref{RGSC}, the extremal example that we have in mind is when $y$ is a large power of 2; if the ternary expansion were roughly uniformly random then we would have $D_2(y) + D_3(y) \approx \frac{2\log y}{3 \log 3}$. Reformulating, what Question \ref{RGSC} is asking is whether, for all sufficiently large $y$, we have $D_2(y)+D_3(y) \asymp \log{y}$. We have some empirical evidence in support of a positive answer; see the appendix.

\subsection{Conditional approximation results}
As we have seen, the estimate \eqref{binary ternary digit changes inequality} given in Lemma \ref{digit changes lemma} plays an important role in obtaining Theorems \ref{main convergence theorem} and \ref{main divergence theorem}. Thus, one possible means of improving those Diophantine approximation results would be to obtain a better lower bound for $D_2(y)+D_3(y).$ In this subsection, we discuss the extent to which such an improved bound would lead to better approximation results.

Suppose that one has the estimate
\[
D_2(y)+D_3(y)\geq h(y),
\]
for an increasing function $h:\mathbb{N}\to (0,\infty)$ which tends to infinity as $y\to\infty$. Then, by repeating the argument in the proof of Lemma \ref{FinalCount}, we find that the conclusions of that lemma are valid for $k\leq \varepsilon h(2^n)$, with a suitable constant $\varepsilon>0.$ Now we make the following replacement of the approximation function:
\[
\widetilde{\psi}: 2^n \mapsto \max\{\psi(2^n), 2^{-\varepsilon h(2^n)}   \}.
\]
Suppose that
\[
\sum_{n = 1}^\infty 2^{-\varepsilon h(2^n)}<\infty.
\]
Then $\sum_{n=1}^{\infty}\widetilde{\psi}(2^n)$ converges if and only if $\sum_{n=1}^{\infty}\psi(2^n)$ converges. 

For $n \in \bN$, let
\[\widetilde{A_n} := \bigcup_{a = 0}^{2^n} B\left(\frac{a}{2^n}, \frac{\widetilde{\psi}(2^n)}{2^n}\right).\]
By a similar argument as in the proof of Theorem \ref{main convergence theorem}, we see that 
\[
\mu(A_n)\leq \mu(\widetilde{A}_n)\ll \widetilde{\psi}(2^n).
\]
Thus, by applying the convergence Borel--Cantelli Lemma (Lemma \ref{first Borel-Cantelli}), we obtain the following result.

\begin{theorem} \label{conditional convergence theorem}
There exists an effectively computable universal constant $\eps > 0$ such that the following holds. If $D_2(y)+D_3(y)\geq h(y)$ for all $y\geq 1$, where $h:\mathbb{N}\to (0,\infty)$ is an increasing function, and
	\[
	\sum_{n = 1}^\infty  (\psi(2^n) + 2^{-\eps h(2^n)})<\infty,
	\]
then
	\[
	\mu(W_2(\psi))=0.
	\]
\end{theorem}

\noindent
We thereby obtain the convergence part of Conjecture \ref{main conjecture} conditionally, for example if we can choose $h(y) = C \log\log y$ for some sufficiently large and effectively computable constant $C > 0$. Recall that we can choose $h(y)\gg \log\log y/\log\log\log y$ for large $y$ unconditionally, by inequality (\ref{binary ternary digit changes inequality}), so in some sense we are quite close to obtaining the convergence part of Conjecture \ref{main conjecture}.

Similarly, we can also improve Theorem \ref{main divergence theorem} if we have a stronger lower bound for $D_2(y) + D_3(y)$.

\vbox{
\begin{theorem} \label{conditional divergence theorem}
Suppose $D_2(y)+D_3(y)\geq h(y)$ for all $y\geq 1$, where $h:\mathbb{N}\to (0,\infty)$ is an increasing function. 
\begin{enumerate} [(a)] 
\item \label{a} If $h(y) \gg \log y$ then for $\psi(2^n)=\frac{1}{n}$ we have $\mu(W_2(\psi))=1$.
\item \label{b} If $h(y)\gg \log\log y$ then for $\psi(2^n)=\frac{1}{1+\log n}$ we have $\mu(W_2(\psi))=1$.
\end{enumerate}
\end{theorem}}

\noindent This can be proved by repeating the arguments in Section \ref{divergence theory proof}. The difference is the choice of $n^+$ in~\eqref{n+}. In Case (\ref{a}), we can choose $n^+=n+\lfloor\varepsilon n\rfloor$, where $\varepsilon>0$ is a constant which depends on the implicit constant in $h(y)\gg \log y.$ In Case (\ref{b}), we can choose $n^+=n+\lfloor\varepsilon \log n\rfloor$ for some $\varepsilon>0.$

\begin{remark} In Case \eqref{a}, which corresponds to a positive answer to Question \ref{RGSC}, the conclusion is almost the divergence part of Conjecture \ref{main conjecture}. Indeed, as $\sum_{n=1}^\infty \frac1{n(\log n)^2}$ converges, the function $\psi(2^n) = 1/n$ comes within a log-power factor of the conjectured truth (or doubly-logarithmic in the input). With the same method, one can show that the same conclusion holds for any approximation function $\psi$ satisfying
	\[
	\limsup_{k\to\infty} \sum_{n=2^k}^{2^{k+1}} \psi(2^n)>0.
	\]
The reason for introducing Case (\ref{b}) is that it appears to be closer to our reach; see Theorem \ref{LangConsequence}.
	\end{remark}

\subsection{Conditional estimates for $D_2(y)+D_3(y)$}
We believe that Question \ref{RGSC} should be difficult to answer. In this section, we illustrate how to sharpen inequality (\ref{binary ternary digit changes inequality}) by assuming the Lang--Waldschmidt Conjecture on Baker's logarithmic sum estimates \cite[Conjecture 1 (page 212)]{Lan1978}.

\begin{conjecture} [Lang--Waldschmidt Conjecture] \label{LangConjecture} Let $a_1, \ldots, a_n, b_1, \ldots, b_n$ be non-zero integers with
	\[
	\Lam := \sum_i b_i \log a_i \ne 0.
	\]
	Then
	\[
	\log |\Lam| \ge - C_n (\log A + \log B),
	\]
	where 
	\[
	A = \max_i |a_i|, \qquad B = \max_i |b_i|,
	\]
	and $C_n>0$ is an effectively computable constant which depends only on $n.$
\end{conjecture}

\begin{theorem} \label{LangConsequence}
	Assuming the Lang--Waldschmidt Conjecture, for sufficiently large $y$, we have
	\[
	D_2(y) + D_3(y) \gg \log \log y.
	\]
\end{theorem}
\begin{remark}
	This asserts that, for some constant $c>0$, we have $D_2(y)+D_3(y)\geq c\log\log y$ for all large enough $y$. This constant $c$ is effectively computable from Conjecture \ref{LangConjecture}. Recalling Theorem \ref{conditional convergence theorem} and the subsequent discussion regarding the constant $C$, we see that if $c\geq C$ then the convergence part of Conjecture \ref{main conjecture} would follow.
\end{remark}
\begin{proof}
We follow Stewart's approach in \cite{Stewart ref}. We may assume that $y\geq 16>e^e$ so that $
\log\log y>1.$ Let
\[
y = \sum_{i =0}^r a_i 2^i = \sum_{i = 0}^t b_i 3^i
\]
be the binary and ternary expansions of $y$, and note that $a_r, b_t \ge 1$. Let $m_1 < m_2 < \cdots < m_k$ mark the binary digit changes of $y$, and let $n_1 < n_2 < \cdots < n_s$ mark the ternary digit changes, i.e. for $j = 1, \dots, k$ we have $a_{m_j} \neq a_{m_j+1}$ and for $\ell = 1, \dots, s$ we have $b_{n_{\ell}} \neq b_{n_{\ell}+1}$ (and there are no other digit changes). Let
\[
2 < \tet_1 < \cdots < \tet_F < t,
\]
with these parameters to be decided later. Observe that if each $[\tet_i, \tet_{i+1})$ contains an element of 
\[
\{ r - m_1, \ldots, r - m_k, t - n_1, \ldots, t - n_s \}
\]
then $k + s \ge F$.

Otherwise, suppose $[\tet_j, \tet_{j+1})$ fails to intersect this set; we will choose the $\tet_j$ in such a way as to contradict this premise. It follows from this assumption that for some $\alpha \in \{0,1\}$ we have
\begin{align*}
y &= \sum_{r - i \notin [\tet_j, \tet_{j+1})} a_i 2^i + \sum_{r - i \in [\tet_j, \tet_{j+1})} \alp 2^i \\
&= \alp (2^{r+1}-1) + \sum_{0 \le r - i < \tet_j} (a_i - \alp) 2^i + \sum_{\tet_{j+1} \leq r-i \leq r} (a_i - \alp) 2^i \\
&= 2^{r- \tet_j} \Bigl( \alp 2^{\tet_j+1} + \sum_{u=1}^{\tet_j} (a_{r-\theta_j+u} - \alp)2^u \Bigr) - \alp +  \sum_{i = 0}^{r- \tet_{j+1}} (a_i - \alp) 2^i,
\end{align*}
and similarly in ternary, for some $\beta \in \{0,1,2\}$,
\[y=3^{t-\theta_j}\left(\frac{\beta}{2} \cdot 3^{\theta_j+1} + \sum_{u=1}^{\theta_j}{(b_{t-\theta_j+u}-\beta)3^u}\right) - \frac{\beta}{2} +\sum_{i=0}^{t-\theta_{j+1}}(b_i-\beta){3^i}.\]
Hence
\begin{align} \label{2y}
2y = A_1 2^{r - \tet_j} + A_2 = B_1 3^{t - \tet_j} + B_2,
\end{align}
where $A_1, A_2, B_1, B_2$ are integers satisfying 
\begin{enumerate}[(i)]
\item{$\displaystyle{2 \cdot 2^{\tet_j} \le A_1 \le 4 \cdot 2^{\tet_j}}$,} 
\item{$\displaystyle{|A_2| < 4 \cdot 2^{r - \tet_{j+1}}}$,}			
\item{$\displaystyle{3^{\tet_j} \le B_1 \le 6 \cdot 3^{\tet_j}}$}, and			
\item{$\displaystyle{|B_2| < 4 \cdot 3^{t - \tet_{j+1}}}$.}
\end{enumerate}

Next, observe that by \eqref{2y} we have
\[
1 = \frac{A_1 2^{r - \tet_j} + A_2}{B_1 3^{t - \tet_j} + B_2} =  R \frac{1+X}{1+Y},
\]
where
\[
R = \frac{A_1 2^{r - \tet_j} }{B_1 3^{t - \tet_j}}, \qquad X = \frac{A_2}{A_1 2^{r - \tet_j}}, \text{ and} \qquad  Y = \frac{B_2}{B_1 3^{t - \tet_j}}.
\]
Using (i)--(iv) from above, and since necessarily we have $\theta_{j+1} > 2$, we compute that
\[
|X| \le \frac{ 4 \cdot 2^{r - \tet_{j+1}} } {2 \cdot 2^r} = 2 \cdot 2^{- \tet_{j+1}} < \frac12
\]
and
\[
|Y| \le
\frac{ 4 \cdot 3^{t - \tet_{j+1}} } {3^t} = 4 \cdot 3^{- \tet_{j+1}} < \frac12.
\]
In particular, we have $R > 0$ and
\begin{align*}
\max \{ R, R^{-1} \} &= \max \Bigl \{ \frac{1+X}{1+Y}, \frac{1+Y}{1+X} \Bigr \} 
                     \le 1 + 4 \max \{ |X|, |Y| \} 
                     < 1 + 16 \cdot 2^{-\tet_{j+1}}.
\end{align*}
We now apply the standard inequality
\[
\log (1+ x) \le x \qquad (x > -1)
\]
with $x = R - 1$ and $R^{-1} - 1$, to obtain 
\begin{equation} \label{logR}
| \log R | = \max \{ \log R, \log (R^{-1}) \} \le 2^{5-\tet_{j+1}}.
\end{equation}

On the other hand, 
\[\log{R} = \log{A_1}+(r-\theta_j)\log{2} - \log{B_1} - (t-\theta_j)\log{3}.\] 
So, assuming Conjecture \ref{LangConjecture} we have
\begin{align} \label{Apply Lang}
\log | \log R | \ge -C_4 (\log \max \{ |A_1|, |B_1| \} + \log (r - \tet_j))
\end{align}
unless $\log R=0.$ We may take $C_4\geq 1.$ The latter, $\log{R} = 0$, would imply that
\[
A_1 2^{r-\theta_j}=B_1 3^{t-\theta_j}.
\]
Hence $3^{t-\theta_j}| A_1$, so $3^{t-\theta_j} \le A_1 \le 2^{\tet_j+2}$, and therefore
\[
2^{r+2} \geq 3^{t-\theta_j}2^{r-\theta_j}.
\]
Observe that $3^{t+1}\geq y\geq 2^r,$ and hence $(t+1)\log{3} \geq r\log{2}$. Thus,
\[
(r+2)\log 2\geq (t-\theta_j)\log 3+(r-\theta_j)\log 2\geq (r\log 2-\log 3)+r\log 2-\theta_j \log 6.
\]
This is only possible if
\[
\theta_j\geq \frac{r\log 2-\log 12}{\log 6}.
\]

We next observe that
\[
\tet_{j+1} - 5 \le - \: \frac{\log | \log R|}{\log 2} \leq 3C_4(\tet_j + \log \log y + 1).
\]
When $\theta_j< \frac{r\log 2-\log 12}{\log 6}$, the first inequality above follows from \eqref{logR}. For the second inequality we apply \eqref{Apply Lang} and make use of the upper bounds given in (i) and (iii) for $A_1$ and $B_1$ respectively, as well as noting that $2^r \leq y$. 

Choosing $\tet_1, \ldots, \tet_F$ so that
\[
\tet_{j+1} - 5 > 3C_4(\tet_j + \log \log y+1), \qquad \theta_j< \frac{r\log 2-\log 12}{\log 6} \qquad (1 \le j \le F)
\]
contradicts our premise that $[\tet_j,\theta_{j+1})$ fails to intersect $\{ r - m_1, \ldots, r - m_k, t - n_1, \ldots, t - n_s \}$.  For example, we can choose $\theta_1 = 3$ and $\theta_{j+1} = 6 + \lfloor 3 C_4 (\theta_j + \log{\log{y}} + 1) \rfloor$ for $j \geq 1$. 
Let $F \in \N$ be the largest integer such that $\theta_F < \frac{r \log{2} - \log{12}}{\log{6}}$. Recalling that $\log\log y>1,$ we see that 
\[
\log\log y \leq \theta_2\leq 6+3C_4(4+\log\log y)\leq  21C_4 \log \log y.
\]
Subsequently, we have
\[
\theta_{j+1}\leq 21C_4 \theta_{j}
\]
for $j\geq 2.$ Thus we see that for $j\geq 2,$
\[
\theta_{j}\leq (21C_4)^{j-1} \log\log y.
\]
In conclusion, we can find
\[
F\geq \frac{\log \left(\frac{r\log 2-\log 12}{\log\log y\log 6} \right)}{\log (21C_4)}\gg \log \log y
\]
many points $\theta_1,\dots, \theta_F \in (2,t) $ such that each $[\theta_j,\theta_{j+1})$ contains an element of 
\[
\{ r - m_1, \ldots, r - m_k, t - n_1, \ldots, t - n_s \}.
\]
This completes the proof.
\end{proof}

Inserting this into Theorem \ref{conditional convergence theorem}, we obtain the following further refinement of the benchmark Proposition \ref{basic proposition}. 

\begin{corollary}\label{Conditional Convergence} Assume the Lang--Waldschmidt Conjecture \ref{LangConjecture}. For some effectively computable constant $\varepsilon > 0$, if
\[
\sum_{n=1}^\infty (n^{-\eps} \psi(2^n)^\gam + \psi(2^n)) < \infty,
\]
then $\mu(W_2(\psi)) = 0$.
\end{corollary}

\begin{remark} 
Corollary \ref{Conditional Convergence} applies in particular to the special case $\varphi_a: 2^n \mapsto n^{-a}$ discussed in the introduction, for some $a < \gam^{-1}$.
\end{remark}

\begin{corollary} \label{exponent} Assume the Lang--Waldschmidt Conjecture, and let $\eps$ be a small positive constant. Then for $\mu$-almost every $\alp$, the inequality
\begin{equation*} 
\| 2^n \alp \| < n^{\eps -\log 3 / \log 2}
\end{equation*}
has at most finitely many solutions $n \in \bN$. 
\end{corollary}

\subsection*{Acknowledgements.} We are indebted to Sanju Velani for suggesting this problem and for helpful discussions. We also thank Yann Bugeaud for a fruitful discussion. DA and SC are also grateful to the University of Cambridge for their hospitality during a productive visit in November 2019, and to P\'{e}ter Varj\'{u} for generously funding this visit.

\Addresses

\newpage
\section*{Appendix}

Here we present empirical data suggesting that Question \ref{RGSC} has a positive answer.
\begin{center} 
	\begin{figure}[H]
		\includegraphics[width=0.65\linewidth]{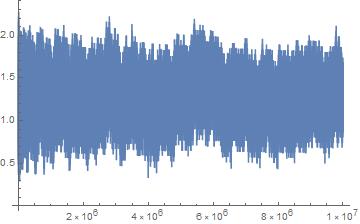}
		\caption{$(D_2(y) + D_3(y))/ \log y$ against $y$}
		\label{data1}
	\end{figure}
\end{center}

This was produced using the software \emph{Mathematica} \cite{Mathematica}, with the following code:

\begin{verbatim}
M = 10^7;
X_; c_; d_; X = ConstantArray[0, M]; 
For[n = 2, n < M + 2, n++,
c = 0; d = 0;
v = IntegerDigits[n, 2]; 
L = Length[v];
For[i = 1, i < L, i++,
If[Part[v, i] != Part[v, i + 1], c++, null]];
v = IntegerDigits[n, 3];
L = Length[v];
For[i = 1, i < L, i++, 
If[Part[v, i] != Part[v, i + 1], d++, null]];
Part[X, n - 1] = (c + d)/Log[n]; 
]
ListLinePlot[X]
\end{verbatim}

We obtain more data by only considering powers of $2$.
\begin{center}
	\begin{figure}[H]
		\includegraphics[width=0.65\linewidth]{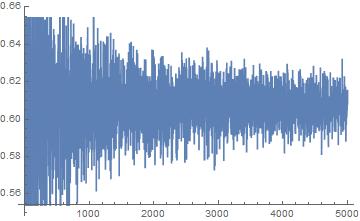}
		\caption{$D_3(2^y)/(y \log 2)$ against $y$}
		\label{data2}
	\end{figure} 
\end{center}

\end{document}